\documentclass[english]{article}
\usepackage[T1]{fontenc}
\usepackage[utf8]{inputenc} \usepackage{float}
\usepackage{amsmath}
\usepackage{amsthm}
\usepackage{amssymb}

\usepackage[
backend=biber,
style=alphabetic,
maxbibnames=10,
maxalphanames=10
]{biblatex}
\input{preamble}

\author{
Sally Dong \\ University of Washington, Seattle \\\texttt{sallyqd@uw.edu}\\
\and
Haotian Jiang  \\ University of Washington, Seattle \\ \texttt{jhtdavid@uw.edu}\\
\and
Yin Tat Lee\thanks{Supported by NSF awards CCF-1749609, DMS-1839116, DMS-2023166, CCF-2105772, a Microsoft Research
Faculty Fellowship, a Sloan Research Fellowship, and a Packard Fellowship} \\ University of Washington, Seattle \\ \texttt{yintat@uw.edu}\\
\and 
Swati Padmanabhan \\ University of Washington, Seattle \\ \texttt{pswati@uw.edu}\\
\and 
Guanghao Ye \\  Massachusetts Institute of Technology\\ \texttt{ghye@mit.edu}
}
\begin{document}

\title{Decomposable Non-Smooth Convex Optimization with Nearly-Linear Gradient Oracle Complexity\blfootnotea{Author names are listed in alphabetical order.}}
\maketitle

\begin{abstract}
Many fundamental problems in machine learning can be formulated by the convex program 
\[
\min_{\vtheta\in \R^d}\ \sum_{i=1}^{n}f_{i}(\vtheta),
\]
where each $f_i$ is a convex, Lipschitz function supported on a subset of $d_i$ coordinates of $\vtheta$.
One common approach to this problem, exemplified by stochastic gradient descent, involves sampling one $f_i$ term at every iteration to make progress.
This approach crucially relies on a notion of uniformity across the $f_i$'s, formally captured by their condition number. In this work, we give an algorithm that minimizes the above convex formulation to $\epsilon$-accuracy in $\widetilde{O}(\sum_{i=1}^n d_i \log (1 /\epsilon))$ gradient computations, with no assumptions on the condition number. 
The previous best algorithm independent of the condition number is the standard cutting plane method, which requires $O(nd \log (1/\epsilon))$ gradient computations. As a corollary, we improve upon the evaluation oracle complexity for decomposable submodular minimization by \cite{axiotis2021decomposable}. 
Our main technical contribution is an adaptive procedure to select an $f_i$ term at every iteration via a novel combination of cutting-plane and interior-point methods.

\end{abstract}

\newpage

\section{Introduction\label{sec:Introduction}}

Many fundamental problems in machine learning are abstractly captured by the convex optimization formulation
\[
\begin{array}{ll}
\mbox{minimize}_{\vtheta\in \R^d} & \sum_{i=1}^{n}f_{i}(\vtheta), 
\end{array}\numberthis\label[none]{eq:generalObj}
\]
where each $f_i$ is a
convex, Lipschitz function.
For example, in empirical risk minimization, each $f_{i}$ measures the loss incurred by the $i$-th data point from the training set. In generalized linear models, each $f_i$ represents a link function applied to a linear predictor evaluated at the $i$-th data point.

The ubiquity of \cref{eq:generalObj} in the setting with smooth $f_i$'s has spurred the development of well-known variants of stochastic gradient methods \cite{robbins1951stochastic,bottou2003large,zhang2004solving, bottou2012stochastic} such as \cite{roux2012stochastic,shalev2013stochastic,johnson2013accelerating,mahdavi2013mixed, defazio2014saga, mairal2015incremental, allen2016improved, hazan2016variance,schmidt2017minimizing}; almost universally, these algorithms leverage the ``sum structure'' of \cref{eq:generalObj} by sampling, in each iteration, one $f_i$ with which to make progress.
These theoretical developments have in turn powered tremendous empirical success in machine learning through widely used software packages such as \texttt{libSVM} \cite{CC01a}.

In many practical applications, \cref{eq:generalObj}  appears with non-smooth $f_i$'s, as well as the additional structure that each $f_i$ depends only on a subset of the problem parameters $\vtheta$. 
One notable example is decomposable submodular function minimization\footnote{In decomposable submodular minimization, each $f_i$ corresponds to the Lov\'asz extension of the individual submodular function and is therefore generally non-smooth.} (SFM), which has proven to be expressive in diverse contexts such as 
determinantal point processes \cite{kulesza2010structured}, 
MAP inference in computer vision \cite{kohli2009robust,vicente2009joint,fix2013structured},
hypergraph cuts \cite{veldt2020minimizing}, and
covering functions \cite{stobbe2010efficient}. 
Another application is found in generalized linear models when the data is high dimensional and sparse. 
In this setting, $f_i$ depends on a restricted subset of the parameters $\vtheta$ that correspond to the features of the data point with non-zero value.
Last but not least, the case with each $f_i$ depending on a small subset of the parameters is also called sparse separable optimization and has applications in sparse SVM and matrix completion \cite{recht2011hogwild}. 

In this work, we initiate a systematic study of algorithms for \cref{eq:generalObj} without the smoothness assumption\footnote{
A function $f$ is said to be \emph{$\beta$-smooth} if $f(y) \leq f(x) + \langle \nabla f(x),  y - x \rangle + \beta/2 \norm{y - x}_2^2$ for all $x,y$ and \emph{$\alpha$-strongly-convex} if $f(y) \geq f(x) + \langle \nabla f(x),  y - x \rangle + \alpha/2 \norm{y - x}_2^2$ for all $x,y$. The condition number of $f$ is defined to be $\kappa = \beta/\alpha$.}. 
Motivated by the aforementioned applications, we introduce the additional structure that each $f_i$ depends on a subset of the coordinates of $\vtheta$. As is standard in the black-box model for studying first-order convex optimization methods, we allow sub-gradient oracle access to each $f_i$. 

\begin{prob}\label[prob]{prob:OurProblemStatement} 
Let $f_1, f_2, \dotsc, f_n:\R^d \mapsto \R$ be convex, $L$-Lipschitz, and possibly non-smooth functions, where each $f_i$ depends on $d_i$ coordinates of $\vtheta$, and is accessible via a (sub-)gradient oracle. Define $m\defeq \sum_{i = 1}^n d_i$ to be the ``total effective dimension'' of the problem.
Let $\vtheta^\star \defeq \arg\min_{\vtheta\in \R^d} \sum_{i = 1}^n f_i (\vtheta)$ be a 
minimizer of \cref{eq:generalObj}, and let $\vtheta^{(0)}$ be an initial point such that $\|\vtheta^{(0)}-\vtheta^\star\|_2\leq R$.
We want to compute a vector $\vtheta\in \R^d$ satisfying
\[
\sum_{i=1}^n f_i (\vtheta) \leq \epsilon LR +  \sum_{i=1}^n f_i (\vtheta^\star). \numberthis\label[none]{eq:optGoal}
\]
\end{prob}

\paragraph{Prior works.} 
We focus on the weakly-polynomial regime and therefore restrict ourselves to algorithms with $\polylog(1/\epsilon)$ gradient oracle complexities.
\cref{tableresults} summarizes the performance of all well-known algorithms applied to \cref{prob:OurProblemStatement}. 
Note that the variants of gradient descent each require bounded condition number.
The results of \cite{nesterov1983method,allen2017katyusha} and cutting plane methods are all complemented by matching lower bounds \cite{woodworth2016tight,DBLP:books/sp/Nesterov04}.

\begin{table}[ht] 
\centering
\resizebox{\linewidth}{!}{
\begin{tabular}{lccc}
\toprule
Authors & Algorithm Type & \thead{Gradient Queries} & \thead{Non-smooth OK?}\\
\midrule
\cite{cauchy1847methode}  & Gradient Descent (GD) & $O(n\kappa \log(1/\epsilon))$  &   \\
\cite{nesterov1983method} & Accelerated (Acc.) GD  & $O(n\sqrt{\kappa} \log(1/\epsilon))$ & \\
\cite{roux2012stochastic,shalev2013stochastic,johnson2013accelerating}  & Stochastic (Stoch.) Variance-Reduced GD & $O((n+\overline{\kappa})\log(1/\epsilon))$ &   \\
\cite{shalev2013accelerated,lin2015universal, frostig2015regularizing, zhang2015stochastic,agarwal2015lower} & Acc. Stoch. Variance-Reduced GD & $O((n+\sqrt{n\overline{\kappa}}) \log(\overline{\kappa}) \log(1/\epsilon))$ &   \\
\cite{allen2017katyusha} & Acc. Stoch. Variance-Reduced GD  &  $O((n+\sqrt{n\overline{\kappa}}) \log(1/\epsilon))$ &  \\ 
\cite{kte88,nn89,vaidya1989new,bv02,lsw15,jlsw20} & Cutting-Plane Method (CPM) & $O(nd \log(1/\epsilon))$ &  \checkmark\\ 
\cite{lee2021tutorial,dly21} & Robust Interior-Point Method (IPM) & $O(\sum_{i=1}^n d_i^{3.5} \log(1/\epsilon))$ &  \checkmark\\ 
\bottomrule
\end{tabular}
}
\caption{Gradient oracle complexities for solving \cref{eq:generalObj} to $\epsilon$-additive accuracy. 
$\kappa$ denotes the condition number of $\sum_i f_i$, and $\overline{\kappa}$ is a variant of the condition number defined to be the sum of smoothness of the $f_i$'s divided by the strong convexity of $\sum_i f_i$.
}
\label{tableresults}
\end{table} 

Even with smooth $f_i$'s, first-order methods perform poorly when the condition number is large, or when there is a long chain of variable dependencies.
These instances commonly arise in applications; an example from signal processing is 
\[
\mbox{minimize}_{\vx} \left\{(\vx_1 -1)^2 + \sum_{i = 2}^{n-1} (\vx_i - \vx_{i+1})^2 + \vx_n^2\right\},  \numberthis\label[none]{eq:Example}
\]
whose variables form an $O(n)$-length chain of dependencies, and has condition number $\kappa = \Theta(n^2)$ and $\bar{\kappa} = \Theta(n^3)$.
Gradient descent algorithms such as \cite{nesterov1983method} and \cite{allen2017katyusha} therefore require $\Omega(n^2)$ gradient queries, despite the problem's total effective dimension being only $O(n)$.

On the other hand, cutting-plane methods (CPM) and robust interior-point methods (IPM) both trade off the dependency on condition number for worse dependencies on the problem dimension.

These significant gaps in the existing body of work motivate the following question:
\begin{quote}
\centering \emph{Can we solve \cref{prob:OurProblemStatement} 
 using a nearly-linear (in total effective dimension) number of subgradient oracle queries?} 
\end{quote}

In this paper, we give an affirmative answer to this question.

\subsection{Our results} 
We present an algorithm to solve \cref{prob:OurProblemStatement} with gradient oracle complexity nearly-linear in the total effective dimension.
\begin{thm}[Main Result (Informal); see \cref{thm:mainResult} for formal statement] \label{thm:main_informal} We give an algorithm that provably solves \cref{prob:OurProblemStatement} using $O(m \log(m/\epsilon))$ subgradient oracle queries, where $m \defeq \sum_{i=1}^n d_i$. 
\end{thm}

Intuitively, the number of gradient queries for each $f_i$ should be thought of as $\widetilde{O}(d_i)$ in our algorithm,
which nearly matches that of the standard cutting-plane method for minimizing the individual function $f_i$.
The nearly-linear dependence on $m$ overall is obtained by leveraging the additional structure on the $f_i$'s and stands in stark
contrast to the $O(nd)$ query complexity of CPM, which is significantly worse in the case where each $d_i \ll d$.
Furthermore, we improve over the current best gradient descent algorithms in the case where the $f_i$'s have a large condition number.

Based on the query complexity of the standard cutting-plane method, we have the following lower bound matching our algorithm's query complexity up to a $\log m$-factor:
\begin{restatable}{thm}{thmLowerBound}
There exist functions $f_1, \dots, f_n : \R^d \mapsto \R$ for which a total of $\Omega(m \log(1/\epsilon))$ gradient queries are required to solve \cref{prob:OurProblemStatement}.
\end{restatable}

An immediate application of \Cref{thm:main_informal} is to decomposable submodular function minimization:

\begin{thm}[Decomposable SFM] \label{thm:decompSFM_intro}
Let $V = \{1, 2, \dots, n\}$, and 
let $F: 2^V \mapsto [-1,1]$ be given by $F(S) = \sum_{i = 1}^n F_i(S \cap V_i)$,  each $F_i: 2^{V_i} \mapsto \mathbb{R}$  a submodular function on $V_i \subseteq V$ with $|V_i| \leq k$. 
We can find an $\epsilon$-additive approximate minimizer of $F$ in $O(n k^2 \log(nk/\epsilon))$ evaluation oracle calls. 
\end{thm} 

\Cref{thm:decompSFM_intro} significantly improves over the evaluation oracle complexity of $\widetilde{O}(n k^6 \log (1/\epsilon))$ given in \cite{axiotis2021decomposable} when the dimension $k$ of each function $F_i$ is large. 
For non-decomposable SFM, i.e. $n = 1$ and $|V_1| = k$, the current best weakly-polynomial time SFM algorithm\footnote{Here, we focus on the weakly-polynomial regime, where the runtime dependence on $\epsilon$ is $\log(1/\epsilon)$.} finds an $\epsilon$-approximate minimizer in time $O(k^2 \log(k/\epsilon))$ \cite{lsw15}. 
Therefore, our result in \Cref{thm:decompSFM_intro} can be viewed as a generalization of the evaluation oracle complexity for non-decomposable SFM in \cite{lsw15}, and the dependence on $k$ in \Cref{thm:decompSFM_intro} might be the best possible. We defer the details of decomposable SFM to \Cref{sec:decompSFM}.

\subsubsection{Limitations}\label[sec]{sec:limitations}
Some limitations of our algorithm are as follows: 
When each $f_i$ depends on the entire $d$-dimensional vector $\vtheta$, as opposed to a subset of the coordinates of size $d_i\ll d$, our gradient complexity simply matches that of CPM. We would like to highlight, though, that our focus is in fact the regime $d_i \ll d$. 
When the $f_i$'s are strongly-convex and smooth, our gradient complexity improves over \cref{tableresults} only when $\kappa$ is large compared to $d_i$. 
Finally, note that we consider only the gradient oracle complexity in our work; 
our algorithm's implementation requires sampling a Hessian matrix and a gradient vector at every iteration, which incur an additional polynomial factor in the overall running time.

\subsection{Technical challenges in prior works}\label[sec]{sec:TechChallenges}
We now describe the key technical challenges that barred existing algorithms from 
solving \cref{prob:OurProblemStatement} in the desired nearly-linear gradient complexity. 

\paragraph{Gradient descent and variants.} As mentioned in \cref{sec:Introduction}, the family of  gradient descent algorithms presented in \cref{tableresults} are not applicable to \cref{prob:OurProblemStatement} without the smoothness assumption. When the objective in \cref{prob:OurProblemStatement} is smooth but has a large condition number, even the optimal deterministic algorithm, Accelerated Gradient Descent (AGD) \cite{nesterov1983method} can perform poorly. For example, when applied to \cref{eq:Example}, AGD updates only one coordinate in each step (thereby requiring $n$ steps), with each step performing $n$ gradient queries (one on each term in the problem objective), yielding a total gradient complexity of $\Omega(n^2)$ \cite{nesterov1983method}. For a similar reason, the fastest randomized algorithm, Katyusha \cite{allen2017katyusha} also incurs a gradient complexity of $\Omega(n^2)$ \cite{woodworth2016tight}. 

\paragraph{Cutting-plane methods (CPM).} Given a convex function $f$ with its set $\mathcal{S}$ of minimizers, CPM minimizes $f$ by maintaining a convex search set $\mathcal{E}^{(k)}\supseteq \mathcal{S}$ in the $k^{\mathrm{th}}$ iteration, and iteratively shrinking $\mathcal{E}^{(k)}$ using the subgradients of $f$. 
Specifically, this is achieved by noting that for any $\vx^{(k)}$ chosen from $\mathcal{E}^{(k)}$, 
if the gradient oracle indicates $\nabla f(\vx^{(k)}) \neq 0$, (i.e.\ $\vx^{(k)}\notin \mathcal{S}$), 
then the convexity of $f$ guarantees $\mathcal{S} \subseteq \mathcal{H}^{(k)} \defeq \left\{\vy: \inprod{\nabla f(\vx^{(k)})}{\vy - \vx^{(k)}} \leq 0 \right\}$, and hence $\mathcal{S} \subseteq \mathcal{H}^{(k)} \cap \mathcal{E}^{(k)}$. 
The algorithm continues by choosing $\mathcal{E}^{(k+1)} \supseteq \mathcal{E}^{(k)} \cap \mathcal{H}^{(k)}$, and different choices of $\vx^{(k)}$ and $\mathcal{E}^{(k)}$ yield different rates of shrinkage of $\mathcal{E}^{(k)}$ until a point in $\mathcal{S}$ is found. 

Solving \cref{prob:OurProblemStatement} via the current fastest CPM \cite{jlsw20} takes $\widetilde{O}(d)$ iterations, each invoking the gradient oracle on every  $f_i$ to compute $\nabla f(\vx^{(k)}) = \sum_{i = 1}^n \nabla f_i(\vx^{(k)})$. 
This results in $\widetilde{O}(nd)$ gradient queries overall, which can be quadratic in $n$ when $d = \Theta(n)$ even if each $f_i$ depends on only $d_i = O(1)$ coordinates. Similar to gradient descent and its variants, the poor performance of CPM on \cref{prob:OurProblemStatement} may therefore be attributed to their inability to query the right $f_i$ required to make progress.

\paragraph{Interior-point methods (IPM).} IPM solves the convex program $\min_{\vu\in \mathcal{S}} \inprod{\vc}{\vu}$ by solving a sequence of unconstrained problems $\min_{\vu} \Psi_{t}(\vu)\defeq \left\{t \cdot \inprod{\vc}{\vu} + \psi_{\mathcal{S}}(\vu)\right\}$ parametrized by increasing $t$, 
where $\psi_{\mathcal{S}}$ is a \emph{self-concordant barrier} function that enforces feasibility by becoming unbounded as it approaches the boundary of the feasible set $\mathcal{S}$. The algorithm starts at $t = 0$, for which an approximate minimizer $\vx_0^\star$ of $\psi_{\mathcal{S}}$ is known, and alternates between increasing $t$ and updating to an approximate minimizer $\vx_t^\star$ of the new $\Psi_{t}$ via Newton's method. 
For a sufficiently large $t$, the minimizer $\vx_t^\star$ also approximately optimizes the original problem $\min_{\vu\in \mathcal{S}} \inprod{\vc}{\vu}$ with sub-optimality gap $O(\nu/t)$, where $\nu$ is the self-concordance parameter of the barrier function used.

To apply IPM to \Cref{prob:OurProblemStatement}, we may first transform \Cref{eq:generalObj} to  $\min_{(\vu,\vz) \in \kcal}  \sum_i \vz_i$, where $\kcal = \{(\vu,\vz): (\vu_i, \vz_i) \in \kcal_i, \forall i \in [n]\}$ is the feasible set. 
Using the universal barrier $\psi_i$ for each $\kcal_i$ \cite{nesterov1994interior}, the number of iterations of IPM is  $\widetilde{O}(\sqrt{\sum_{i=1}^n d_i})$, each requiring the computation of the Hessian and gradient of $\psi_i$ for all $i\in [n]$, leading to a total of $\widetilde{O}(n^{1.5})$ sub-gradient queries to $f_i$'s even when all $d_i = O(1)$. 
Even when leveraging the recent framework of robust IPM for linear programs \cite{lee2021tutorial}, the computation of each Hessian (by sampling the corresponding $\kcal_i$ \cite{halv21}) yields a total sub-gradient oracle complexity of $\widetilde{O}(\sum_{i=1}^n d_i^{3.5})$, far from the complexity we seek.

\subsection{Our algorithmic framework}\label[sec]{sec:OurAlgFrm} We now give an overview of the techniques developed in this work to overcome the above barriers. 
First, we transform \cref{eq:generalObj} into a convex program over structured convex sets: 
\[
\begin{array}{ll}
    \mbox{minimize} & \inprod{\vc}{\vx},  \\
     \mbox{subject to} & \vxi\in\ki\subseteq \R^{d_i + 1} \;\forall i\in[n]\\
     &  \ma\vx=\vb.
\end{array}\numberthis\label[none]{eq:TechOutlineReducedProb}
\]
\\where $\vx$ is the concatenation of the vectors $\vx_1, \dots, \vx_n$, and notably the convex sets $\ki$ are all disjoint.
Under this transformation we do not have explicit, closed-form expressions for each $\ki$; however, the subgradient oracle for $f_i$ can be transformed equivalently to a separation oracle $\ki$. 
We define $\kcal \defeq \kcal_1 \times \kcal_2 \times \dotsc \times \kcal_n$.

\paragraph{Main idea: combining CPM and IPM.} Recall that CPM maintains a convex set which initially contains the feasible region and gradually shrinks around the minimizer, while IPM maintains a point inside the feasible region that moves toward the minimizer. 
Our novel idea is to combine both methods and maintain \emph{an inner convex set $\kini$ as well as an outer convex set $\kouti$} for each $i \in [n]$, such that $\kini \subseteq \ki\subseteq \kouti$. 
We define $\kin$ and $\kout$ analogously to $\kcal$.
When  \cref{eq:Condition1} and \cref{eq:Condition2} are satisfied for all $i\in [n]$, we make IPM-style updates without needing to make any oracle calls.
When \cref{eq:Condition2} is violated for some $i \in [n]$, we query the separation oracle at the point $\vxosi$ defined as the centroid of $\kouti$ (c.f. \cref{prop:EquivalenceOfCentroidAndMinimizer}). 
Based on the oracle's response, we iteratively either grow $\kini$ (and, thus, $\kin$) outward or shrink $\kouti$ (and, thus, $\kout$) inward, until ultimately they approximate $\kcal$ around the optimum point.

\paragraph{First benefit: large change in volume.} 
If the point $\vxosi$ violates \cref{eq:Condition2} for some $i\in[n]$, we query the separation oracle to see if $\vxosi\in \ki$ or not. 
If $\vxosi\in \ki$, then it is used to expand $\kini$, yielding in a large volume increase for $\kini$. 
On the other hand, if $\vxosi\notin\ki$, the fact that it is the centroid of $\kouti$ results in
a large volume decrease for $\kouti$ when it is intersected with a halfspace through $\vxosi$. 
Thus, our algorithm witnesses a large change in volume of one of $\kini$ and $\kouti$, regardless of the answer from the oracle. Just like in standard CPM, this rapid change in volume is crucial to achieving the algorithm's oracle complexity. 

\paragraph{Second benefit: making a smart choice about querying $f_i$.} Since the algorithm maintains both an inner and outer set approximating $\kcal$, by checking if $\kini$ and $\kouti$ differ significantly (\cref{eq:Condition2} essentially performs this function), we can determine if $\ki$ is poorly approximated, and if so, improve the inner and outer approximations of the true feasible set. 
Choosing the right $\ki$ translates to choosing the right $f_i$ to make progress with at an iteration; thus, we address the central weakness of the gradient descent variants in solving \cref{eq:generalObj}.

\section{Notation and preliminaries\label{sec:Notation-and-Preliminaries}}

We lay out the notation used in our paper as well as the definitions
and prior known results that we rely on. We use lowercase
boldface letters to denote (column) vectors and uppercase boldface letters
to denote matrices. We use $\vxi$ to denote the $i^{\mathrm{th}}$
 block of coordinates in the vector $\vx$ (the ordering of these blocks is not important in our setup).  We use $\succeq$ and $\preceq$ to denote the Loewner ordering of matrices. 

We use $\inprod{\vx}{\mathbf{\vy}}$ to mean the Euclidean inner product $\vx^\top \vy$.
A subscript $\vx$ in the inner product notation means it is induced by the
Hessian of some function (which is clear from context) at $\vx$; for example,
$\inprod{\mathbf{u}}{\mathbf{v}}_{\vx}= \vu^\top \nabla_{ii}^{2}\barr(\vx) \vv$ with $\barr$ inferred from context. We define the local norm of $\vv$ at $\vx$ analogously:
$\|\vv\|_{\vx}=\sqrt{\inprod{\vv}{\nabla^{2}\barr(\vx)\cdot\vv}}.$ We also define the norm $\|\vv\|_{\vx, 1} \defeq \sum_{i = 1}^n \|\vv\|_{\vxi}$. 

We use $\psi$ to represent barrier functions and $\pot$ to represent
potential functions, with appropriate subscripts and superscripts
to qualify them as needed.

\subsection{Facts from convex analysis}

In this section, we present some definitions and properties from convex
analysis that are useful in our paper. These results are standard
and may be found in, for example, \cite{rockafellar1970convex,boyd2004convex}. 
\begin{defn}
\label[defn]{defn:Conjugate} Let $f:\R^{n}\rightarrow\R.$ Then the function
$f^{\ast}:\Rn\rightarrow\R$ defined as 
\[
f^{\ast}(\vy)=\sup_{\vx\in\textrm{dom}(f)}\left[\inprod{\vx}{\vy}-f(\vx)\right]
\]
is called the Fenchel conjugate of the function $f.$ An immediate consequence
of the definition (and by applying the appropriate convexity-preserving
property) is that $f^{\ast}$ is convex, regardless of the convexity
of $f.$ We use the superscript $\ast$ on functions to denote their
conjugates. 
\end{defn}

\begin{lem}[Biconjugacy]
\label[lem]{lem:fastast}  For a closed, convex function $f,$ we have
$f=f^{\ast\ast}.$ 
\end{lem}

\begin{lem}[\cite{rockafellar1970convex}]
\label[lem]{lem:GradConjugate} For a closed, convex differentiable function
$f,$ we have $\vy=\nabla f(\vx)\iff\vx=\nabla f^{\ast}(\vy).$ 
\end{lem}

\begin{lem}[\cite{rockafellar1970convex}]
\label[lem]{lem:HessianConjugate} For a strictly convex, twice-differentiable
function $f$, we have $\nabla^{2}f^{\ast}(\nabla f(\vx))=(\nabla^{2}f(\vx))^{-1}.$ 
\end{lem}

\ifdefined\isneurips
\defPolar*
\lemPolarIntersection*
\else 
\begin{restatable}[Polar of a Set \cite{rockafellar2009variational}]{defn}{defPolar}
\label[defn]{def:Polar} Given a set $\mathcal{S}\subseteq\Rn,$ its polar
is defined as 
\[
\mathcal{S}^{\circ}\defeq\left\{ \vy\in\Rn:\inprod{\vy}{\vx}\leq1,\text{ \ensuremath{\forall\vx\in\mathcal{S}}}\right\} .
\]
\end{restatable}

\begin{restatable}[\cite{rockafellar1970convex}]{lem}{lemPolarIntersection}
\label[lem]{lem:polarAfterAddingPoint} 
Let $\mathcal{S}\subseteq\Rn$ be a closed, compact, convex
set, and let $\vy$ be a point. 
Then $(\conv\left\{ \mathcal{S},\vy\right\} )^{\circ}\subseteq\mathcal{S}^{\circ}\cap\h$,
where $\h$ is the halfspace defined by $\h=\left\{ \vz\in\R^{n}:\inprod{\vz}{\vy}\leq1\right\}$.
\end{restatable}
\fi 

\subsection{Background on interior-point methods}

Our work draws heavily upon geometric properties of self-concordant
functions, which underpin the rich theory of interior-point methods. 
We list below the formal results needed for our analysis,
and refer the reader to \cite{nesterov1994interior, renegar2001mathematical} for a detailed exposition of this function
class. 
We begin with the definitions of self-concordant functions and self-concordant barriers:
\ifdefined\isneurips
\defScBarr*
\else 
\begin{restatable}[Self-concordance \cite{nesterov1994interior}]{defn}{defScBarr}
\label[defn]{def:SelfConcordanceBarr}
A function $F:Q\mapsto\R$ is a self-concordant function on a convex set $Q$ if for any $\vx\in Q$ and any $\vh$, 
\[
\lvert D^3 F(\vx)[\vh, \vh, \vh]\rvert \leq 2(D^2 F(\vx)[\vh, \vh])^{3/2}, 
\] 
where $D^k F(\vx)[\vh_1, \dotsc, \vh_k]$ denotes the $k$-th derivative of $F$ at $\vx$ along the directions $\vh_1,\dotsc,\vh_k$. We say $F$ is a $\nu$-self-concordant barrier if $F$ further satisfies $\nabla F(x)^\top (\nabla^2 F(x))^{-1} \nabla F(x) \leq \nu$ for any $\vx\in Q$.
\end{restatable}
\fi

\begin{thm}[Theorem 2.3.3 from \cite{renegar2001mathematical}]
\label{thm:sc1} If $f$ is a self-concordant barrier, then for all
$\vx$ and $\vy\in\textrm{dom}(f)$, we have $\inprod{\nabla f(\vx)}{\vy-\vx}\leq\nu,$ where
$\nu$ is the self-concordance of $f$. 
\end{thm}

\begin{thm}[Theorem 2.3.4 from \cite{renegar2001mathematical}]
\label{thm:sc2} If $f$ is a $\nu$-self-concordant barrier such that $\vx,\vy\in\textrm{dom}(f)$
satisfy $\inprod{\nabla f(\vx)}{\vy-\vx}\geq0$, then $\vy\in\ball_{\vx}(\vx,4\nu+1).$
\end{thm}

We now state the following result from self-concordance calculus. 
\begin{thm}[Theorem 3.3.1 of \cite{renegar2001mathematical}]
\label{thm:ConjugateSC} If $f$ is a (strongly nondegenerate) self-concordant
function, then so is its Fenchel conjugate $f^{\ast}.$ 
\end{thm}

The following result gives a bound on the quadratic approximation of a function, 
with the distance between two points measured in the local norm. The convergence
of Newton's method can be essentially explained by this result. 
\begin{thm}[Theorem 2.2.2 of \cite{renegar2001mathematical}]
\label{thm:QuadApproxErr} If $f$ is a self-concordant function,
$\vx\in\textrm{dom}(f)$, and $\vy\in\ball_{\vx}(\vx,1)$, then 
\[
f(\vy)\leq f(\vx)+\inprod{\nabla f(\vx)}{\vy-\vx}+\frac{1}{2}\|\vy-\vx\|_{\vx}^{2}+\frac{\|\vy-\vx\|_{\vx}^{3}}{3(1-\|\vy-\vx\|_{\vx})},
\]
where $\|\vy-\vx\|_{\vx}^{2}\defeq\inprod{\vy-\vx}{\nabla^{2}f(\vx)\cdot(\vy-\vx)}$.
\end{thm}

Finally, we need the following definitions of entropic barrier and universal barrier. 

\ifdefined\isneurips
\defEntropicBarr*

\defUniBarrier*
\else

\begin{restatable}[\cite{bubeck2015entropic, chewi2021entropic}]{defn}{defEntropicBarr}
\label[defn]{def:EntropicBarr} Given a convex body $\kcal\subseteq\R^{n}$ and some fixed $\theta\in\R^{n}$,  define
the function $f(\theta)=\log\left[\int_{\vx\in\kcal}\exp\inprod{\vx}{\theta}d\vx\right]$.  Then the Fenchel conjugate $f^{\ast}:\textrm{int}(\kcal)\rightarrow\R$ is a self-concordant barrier termed the \emph{entropic barrier}. The entropic barrier is $n$-self-concordant. 
\end{restatable}

\begin{restatable}[\cite{DBLP:books/daglib/0071613,lee2021universal}]{defn}{defUniBarrier}
\label[defn]{def:UniversalBarr} Given a convex body $\kcal\subseteq\R^{n}$, the \emph{universal barrier} of $\kcal$ is defined as $\psi:\operatorname{int}(\kcal)\to \R$ by 
\[
\psi(\vx) = \log\vol((\kcal-\vx)^\circ)
\]
where $(\kcal-\vx)^\circ=\{\vy\in \R^n:\vy^\top(\vz-\vx)\leq 1,\forall \vz\in \kcal\}$ is the polar set of $\kcal$ with respect to $\vx$. The universal barrier is $n$-self-concordant.
\end{restatable}

\fi 

\subsection{Facts from convex geometry}

Since our analysis is contingent on the change in the volume of convex
bodies when points are added to them or when they are intersected
with halfspaces, we invoke the classical result by Gr\"unbaum several
times. 
We therefore state its relevant variants next: 
\cref{thm:GrunbaumGeneral} applies to
log-concave distributions, and \cref{cor:GrunbaumVolApprox} is its
specific case, since the indicator function of a convex set is a log-concave
function \cite{boyd2004convex}. 

\ifdefined\isneurips
\thmGrunBaumMain*
\else 

\begin{restatable}[\cite{grunbaum1960partitions, bubeck2020chasing}]{thm}{thmGrunBaumMain}
\label{thm:GrunbaumGeneral} Let $f$ be a log-concave distribution
on $\R^{d}$ with centroid $\vc_{f}.$ Let $\mathcal{H}=\left\{ \vu\in\R^{d}:\inprod{\vu}{\vv}\geq q\right\} $
be a halfspace defined by a normal vector $\vv\in\R^{d}$. Then, $
\int_{\mathcal{H}}f(\vz)d\vz\geq\frac{1}{e}-t^{+},$ 
where $t=\frac{q-\inprod{\vc_{f}}{\vv}}{\sqrt{\mathbb{E}_{\vy\sim f}\inprod{\vv}{\vy-\vc_{f}}^{2}}}$
is the distance of the centroid to the halfspace scaled by the standard
deviation along the normal vector $\vv$ and $t^+ \defeq \max\{0, t\}$.  
\end{restatable}
\fi 

\begin{rem}\label[rem]{rem:GrunbaumSimpleCase}
A crucial special case of \cref{thm:GrunbaumGeneral} is that cutting a convex set through its centroid yields two parts, the smaller of which has volume at least $1/e$ times the original volume and the larger of which is at most $1-1/e$ times the original total volume. 
\end{rem}

\begin{cor}[\cite{grunbaum1960partitions}]
\label[cor]{cor:GrunbaumVolApprox} Let $\kcal$ be a convex set with centroid
$\mu$ and covariance matrix $\Sigma.$ Then, for any point $\vx$
satisfying $\|\vx-\mu\|_{\Sigma^{-1}}\leq\eta$ and a halfspace $\h$
such that $\vx\in\h$, we have $\vol(\kcal\cap\h)\geq\vol(\kcal)\cdot(1/e-\eta).$
\end{cor} 

Finally, we need the following facts. 
\begin{fact}[Volumes of standard objects]\label[fact]{fact:volSphereCone}
The volume of a $q$-dimensional Euclidean ball is given by $\vol(\ball_{q}(0, \bar R))= \frac{\pi^{q/2}}{\Gamma(1+ q/2)} {\bar R}^q$,  and the $\textrm{volume of a $q$-dimensional cone} = \frac{1}{q+1}\cdot \textrm{volume of base} \cdot \textrm{height}.$
\end{fact}

\section{Our algorithm \label[sec]{sec:AlgOverview}}

We begin by reducing \cref{prob:OurProblemStatement} to the following slightly stronger formulation (see \cref{thm:mainResult} for the detailed reduction):

\[
\begin{array}{ll}
    \mbox{minimize} & \inprod{\vc}{\vx},  \\
     \mbox{subject to} & \vxi\in\ki\subseteq \R^{d_i + 1} \;\forall i\in[n]\\
     &  \ma\vx=\vb.
\end{array}\numberthis \label[prob]{eq:1main}
\]
where $\vx$ is a concatenation of vectors $\vx_i$'s, and the $\ki$'s are disjoint convex sets.
This formulation decouples the overlapping support of the original $f_i$'s by introducing additional variables tied together through the linear system $\ma\vx=\vb$. 
Each $\ki$ is constructed by applying a standard epigraph trick to the function $f_i$.

Note that we do not have a closed-form expression for $\ki$. 
Instead, the subgradient oracle for $f_i$ translates to a \emph{separation oracle} for $\ki$: 
on a point $\vzi$ queried by the oracle, the oracle either asserts $\vzi\in\ki$, or returns a separating
hyperplane that separates $\vzi$ from $\ki$. 
 
 At the start of our algorithm, we have the following guarantee:
\begin{lem}\label[lemma]{assumption:alg-assumption} At the start of our algorithm, we are guaranteed the existence of the following. 
\begin{enumerate}[wide]
\item Explicit convex sets $\kin\defeq \kcal_{\textrm{in,}1}\times \kcal_{\textrm{in,}2} \times \cdots \times \kcal_{\textrm{in,}n}$ and $\kout\defeq \kcal_{\textrm{out,}1}\times \kcal_{\textrm{out,}2} \times \cdots \times \kcal_{\textrm{out,}n}$ 
such that $\kin\subseteq \kcal \defeq \kcal_1 \times \cdots \times \kcal_n \subseteq \kout$, 
\item An 
initial $\vx_{\textrm{initial}}\in \kin$  such that $\ma\vx_{\textrm{initial}} = \vb$.
\end{enumerate}
\end{lem}
 
We show how to construct such a set $\kin$ in \cref{sec:initial-ki} and how to find such a $\kout$ and $\vx_{\textrm{initial}}$ in \cref{sec:initial-point-reduction}.
  
\subsection{Details of our algorithm}
In this section, we explain our main algorithm (\cref{alg:decomposable-fn-min}). 

The inputs to \cref{alg:decomposable-fn-min} are: 
initial sets $\kin$ and $\kout$ satisfying $\kin\subseteq\kcal\subseteq\kout$, 
an initial point $\vx\in \kin$ satisfying $\ma\vx = \vb$, a 
separation oracle $\oi$ for each $\ki$, the objective vector $\vc$, 
and scalar parameters $m$, $n$, $R$, $r$, and $\epsilon$. 
All the parameters are set in the proof of \cref{thm:MainThmOfKiProblem}.

Throughout the algorithm, we maintain a central path parameter $t$ for IPM-inspired updates, the current solution $\vx$,
and convex sets $\kini$ and $\kouti$ satisfying $\kini\subseteq \ki\subseteq\kouti$ for each $i\in [n]$. To run IPM-style updates, we choose the entropic barrier on $\kout$ and the universal barrier on $\kin$.

Given the current set $\kout$, the current $t$, and the entropic barrier $\barrout$ defined on $\kouthat\defeq \kout\cap\left\{\vu:\ma \vu = \vb\right\}$, we define the point 
\begin{equation}
\vxos\defeq\arg\min_{\vx\in \kouthat}\left\{ t\inprod{\vc}{\vx}+\barrout(\vx)\right\}. \label{eq:xoutstar-min}
\end{equation} 
Per the IPM paradigm, for the current value of $t$, this point serves as a target to ``chase'' when optimizing $\inprod{\vc}{\vx}$ over the set $\kouthat$. 
Although our overall goal in \cref{eq:1main} is to optimize over $\kcal\cap \left\{\vu: \ma\vu=\vb\right\}$, 
we do not know $\kcal$ explicitly and therefore must use its known proxies, $\kin$ or $\kout$; we choose $\kout$ because $\kout\supseteq\kcal$ ensures we do not miss a potential optimal point. 

Having computed the current target $\vxos$, we move the current solution $\vx$ towards it by taking a Newton step, provided certain conditions of feasibility and minimum progress are satisfied. 
If either one of these conditions is violated, we update either $\kin$, $\kout$, or the parameter $t$. 

\begin{algorithm}
\caption{Minimizing Decomposable Convex Function}
\label{alg:decomposable-fn-min}

\begin{algorithmic}[1]

\algnewcommand{\LeftComment}[1]{\State \(\triangleright\) #1} 

\LeftComment{ Solving \cref{eq:1main} } 
\State \textbf{Input.} $\epsilon$, $\ma$, $\vb$, $\vc$, $R$, $r$, $m$
, $n$, $\vx$, $\kin$, $\kout$, and $\oi$ for each $i\in [n]$. 
\State \textbf{Initialization.} $t =  \frac{m\log m}{\sqrt{n} \|\vc\|_2 R}$ and $\tend=\frac{8m}{\epsilon\|\vc\|_{2}R}$, $\eta =\frac{1}{100}$, and $\vxos$ (via \cref{eq:xoutstar-min})

\While{true} \label{line:ForLoop}

	\If{$\inprod{\vc}{\vx}\leq\inprod{\vc}{\vxos}+\frac{4 m}{t}$} \Comment{
Either update $t$ or end the algorithm }\label{line:updateT}

        \If {$t\geq \tend$}\label{line:EndAlg}
        
            \State \Return $\arg\min_{\vx:\vx\in\kin,\ma\vx=\vb}\left\{ t\inprod{\vc}{\vx}+\sum_{i=1}^{n}\barrini(\vxi)\right\}$. \Comment{End the algorithm}         
            
        \EndIf

		\State $t\leftarrow t\cdot\left[1+\frac{\eta}{4 m}\right]$ \Comment{Update $t$}
\label{line:updateT-Rule}

		\State Update $\vxos$ and jump to \cref{line:ForLoop}\Comment{$\vxos$ computed as as per \cref{eq:xoutstar-min}}

	\EndIf \label{line:updateT-end}

	\For{ all $i\in [n]$ } \label{line:KinKoutUpdate}

		\If{ $\inprod{\nabla\barrini(\vxi)}{\vxosi-\vxi}+\eta \|\vxosi-\vxi\|_{\vxi}\geq4 d_{i}$} 

			\If{ $\vxosi\in\ki$} \Comment{ Query $\oi$
}

				\State $\kini=\conv(\kini,\vxosi)$  \Comment{ Update $\kini$}\label{line:KinUpdated}

			\Else
			
			\State $\kouti=\kouti\cap\hi$, where $\hi=\oi(\vxosi)$  \Comment{
Update $\kouti$}\label{line:KoutUpdated}

			\EndIf

			\State Update $\vxos$ and jump to \cref{line:ForLoop} \Comment{$\vxos$ computed as per \cref{eq:xoutstar-min}}

		\EndIf

	\EndFor \label{line:KinKoutUpdateEnd}

	\State Set $\delta_{\vx}\defeq\frac{\eta}{2}\cdot\frac{\vxos-\vx}{\|\vxos-\vx\|_{\vx,1}},$where
$\|\mathbf{u}\|_{\vx,1}\defeq\sum_{i=1}^{n}\|\mathbf{u}\|_{\vxi}.$\label{line:IPM-step}

	\State $\vx\leftarrow\vx+\delta_{\vx}$ \Comment{Move $\vx$ towards $\vxos$} \label{line:MoveX}

\EndWhile

\end{algorithmic}
\end{algorithm}

\paragraph{Updating $\protect\vx$.} 
In order to move $\vx$ towards $\vxos$, we require two conditions to hold: 
$\vxos \in \kin$ and $\inprod{\vc}{\vx} \geq \inprod{\vc}{\vxos} + O(1/t)$. 

The first condition implies $\vxos\in \kcal$, 
which would in turn ensure feasibility of the resulting $\vx$ after a Newton step.
To formally check this condition, we check if the following inequality is satisfied for all $i\in [n]$ and for a fixed constant $\eta$: 
\begin{equation}
\inprod{\nabla\barrini(\vxi)}{\vxosi-\vxi}+\eta \cdot\|\vxosi-\vxi\|_{\vxi}\leq 4 d_{i}. \label[ineq]{eq:Condition2}
\end{equation}
The intuition is that since any point within the domain of a self-concordant barrier satisfies the inequalities in \cref{thm:sc1} and \cref{thm:sc2}, 
violating \cref{eq:Condition2}
implies that $\vxosi$ is far from $\kini$, and as a result, $\vxos$ is not a good candidate to move $\vx$ towards.

The second condition we impose, one of ``sufficient suboptimality'',  ensures significant progress in the objective value can be made when updating $\vx$. 
Formally, we check if
\begin{equation}
\vc^{\top}\vxos+\frac{4 m}{t}\leq\vc^{\top}\vx.\label[ineq]{eq:Condition1}
\end{equation}
If the inequality holds, then there is still ``room for progress'' to lower the value of $\inprod{\vc}{\vx}$ by updating $\vx$; if the inequality is violated, we update $t$ instead. 

Given the two conditions hold, we move $\vx$ towards $\vxos$ in \cref{line:MoveX}. 
The update step is normalized by the distance between $\vx$ and $\vxos$ measured in the local norm,
which enforces $\vx\in\kcal$ 
(since by the definition of self-concordance, the unit radius Dikin ball lies
inside the domain of the self-concordance barrier), 
and also helps bound certain first-order error
terms (\cref{eq:potential-ipm-1} in \cref{sec:IPM-analysis}).

The rest of this section details the procedure for when either of these conditions is violated. 

\paragraph{Updating the inner and outer convex sets.}

Suppose \cref{eq:Condition2} is violated for some $i\in [n]$.
Then $\vxosi\notin\kini$, which in turn means $\vxos$ \emph{might} not be in the feasible set $\kcal$. 
To reestablish \cref{eq:Condition2} for $i$, we can either update $\kini$, or update $\kouti$ and compute a new $\vxosi$ by \cref{eq:xoutstar-min}.

To decide which option to take, 
we query $\oi$ at the point $\vxosi$: 
if the oracle indicates that $\vxosi\in\ki$, 
then we incorporate $\vxosi$ into $\kini$ by redefining $\kini = \textrm{conv}(\kini, \vxosi)$ to be the convex hull of the current $\kini$ and $\vxosi$ (\cref{line:KinUpdated}). 
If, on the other hand, $\vxosi\notin \ki$, the oracle $\oi$ will return a halfspace $\mathcal{H}_i$ satisfying $\hi \supseteq \ki$.
Then we redefine $\kouti = \kouti \cap \mathcal{H}_i$ (\cref{line:KoutUpdated}).
After processing this update of the sets, the algorithm recomputes $\vxos$ and returns to the main loop since updating the sets does not necessarily imply that the new $\vxos$ satisfies $\vxos\in \kin$.

This update rule for the sets is exactly where our novelty lies: \emph{we do not arbitrarily update sets, rather, we update one only after checking the very specific condition $\vxosi\notin \kini$}. Since the separation oracle is called only in this part of the algorithm, performing this check enables us to dramatically reduce the number of calls we make to the separation oracle, thereby improving our oracle complexity. 

Further, this update rule shows that even when we cannot update the current $\vx$, 
we make progress by using all the information from the oracles. 
Over the course of the algorithm, we gradually expand $\kin$ and shrink $\kout$, until they well-approximate $\kcal$. 
To formally quantify the change in volume due to the above operations, we consider the following alternative view of $\vxos$.

\begin{prop}[Section 3 in \cite{bubeck2015entropic}; Section 3 of \cite{klartag2006convex}]\label[prop]{prop:EquivalenceOfCentroidAndMinimizer}
 Let $\theta\in\R^{n}$, and let $p_{\theta}$ be defined
as $p_{\theta}(\vx)\defeq\exp(\inprod{\theta}{\vx}-f(\theta))$,
where $f(\theta)\defeq\log\left[\int_{\kcal}\exp(\inprod{\theta}{\vu})d\vu\right]$.
Then, 
\[
\mathbb{E}_{\vx \sim p_{\theta}}[\vx]=\arg\min_{\vx\in\mathrm{int}(\kcal)}\left\{ f^{\ast}(\vx) - \inprod{\theta}{\vx}\right\}.
\]
\end{prop} 
By this proposition, $\vxos$ defined in \cref{eq:xoutstar-min} satisfies 
\begin{equation}
\vxos\defeq\mathbb{E}_{\vx \sim\exp\left\{ -t\inprod{\vc}{\vx}-\log\left[\int_{\kouthat}\exp(-t\inprod{\vc}{\vu})d\vu\right]\right\} }[\vx],\label{eq:xoutstar-mean}
\end{equation}
that is, $\vxos$ is the centroid of some exponential distribution over $\kouthat$.
As a result, if $\vxosi\notin \ki$, the hyperplane cutting $\kouthat$
through $\vxos$ will yield a large decrease in volume of $\kouthat$,
per \cref{rem:GrunbaumSimpleCase}. Therefore, the query result
in a large change in volume in either $\kin$ or $\kout$, allowing us to approximate $\kcal$ with a bounded number of iterations. 

\paragraph{Updating $\protect t$.}  If \cref{eq:Condition1} is violated,
then the current $\vx$ is ``as optimal as one can get'' for the current parameter $t$. 
This could mean one of two things:

The first possibility is that we have already reached an approximate optimum, which we verify by checking whether $t\geq O(1/\epsilon)$ in \cref{line:EndAlg}: 
If true, this indicates that we have attained our desired suboptimality, and the algorithm  terminates by returning 
\[
\vx^{\mathrm{ret}} = \arg\min_{\vx:\vx\in \kin, \ma\vx=\vb} \left\{t \cdot \inprod{\vc}{\vx} + \sum_{i = 1}^n \barrini(\vxi)\right\}.
\]
The point $\vx^{\mathrm{ret}}$ is feasible because it is in $\kin$ by definition, and the suboptimality of $O(1/\tend) = O(\epsilon)$ ensures it is an approximate optimum for the original problem. 

The second possibility is that we need to increase $t$ to set the next ``target suboptimality''. The value of $t$ is increased by a scaling factor of $1 + O(1/m)$ in \cref{line:updateT-Rule}. This scaling factor ensures, like in the standard IPM framework, that the next optimum is not too far from the current one. 
Following the update to $t$, we recompute $\vxos$ by \cref{eq:xoutstar-min}.
Since $\inprod{\vc}{\vx} > \inprod{\vc}{\vxos} + O(1/t)$ is not guaranteed with the new $t$ and $\vxos$, the algorithm jumps back to the start of the main loop.

\section{Our analysis \label{subsec:PotentialFns}} 
To analyze \cref{alg:decomposable-fn-min}, we define the following potential function
that captures the changes in $\kini$, $\kouti$, $t$, and $\vx$
in each iteration:
\begin{equation}
\pot\defeq \underbrace{t\inprod{\vc}{\vx} + \log\left[\int_{\kouthat}\exp(-t\inprod{\vc}{\vu})d\vu\right]}_{\text{entropic terms}}+\underbrace{\sum_{i\in[n]}\barrini(\vxi)}_{\text{universal terms}},\label{eq:TotalPotential}
\end{equation}
where $\log\left[\int_{\kouthat}\exp(-t\inprod{\vc}{\vu})d\vu\right]$ is related
to the entropic barrier on $\kouthat$ (see \cref{sec:OuterPotChange})
and $\barrini$ is the universal barrier on $\ki$. 
In the subsequent
sections, we study the changes in each of these potential functions along with obtaining bounds on the initial and final potentials 
and combine them to bound the algorithm's separation oracle complexity. 

\subsection{Potential change for the entropic terms \label{sec:OuterPotChange}}

In this section, we study the changes in the entropic terms of \cref{eq:TotalPotential} upon updating the outer convex set $\kouthat$ as well as $t.$ These two changes are
lumped together in this section because both updates affect the term $\log\left[\int_{\kouthat}\exp(-t\cdot\inprod{\vc}{\vx})\ d\vx\right]$, albeit in different ways: the update in $\kouthat$ affects it via Gr\"unbaum's Theorem; the update in $t$ affects it via the fact that, by duality with respect to the entropic barrier (\cref{def:EntropicBarr}), $\log\left[ \int_{\vx\in \kouthat} \exp(\inprod{\vx}{\vtheta}) d\vx\right]$ is also self-concordant. We detail these two potential changes below. 

\begin{lem}[Potential analysis for outer set]
\label[lem]{lem:PotKOutChange} Let $\kouthat\defeq\left\{ \vx:\vxi\in\kouti\cap\left\{ \vy:\ma\vy=\vb\right\} \right\} ,$
and let $\pot=\log\left[\int_{\kouthat}\exp(-t\inprod{\vc}{\vu})d\vu\right]+t\inprod{\vc}{\vx}+\sum_{i\in[n]}\barrini(\vxi)$.
Let $\hi$ be the halfspace generated by the separation oracle $\oi$ queried at $\vxosi$ as shown 
in \cref{line:KoutUpdated} of \cref{alg:decomposable-fn-min}. Then the new potential $\potnew=\log\left[\int_{\kouthat\cap\hi}\exp(-t\inprod{\vc}{\vu})d\vu\right]+t\inprod{\vc}{\vx}+\sum_{i\in[n]}\barrini(\vxi)$
is bounded from above as follows. 
\[
\potnew\leq\pot+\log(1-1/e).
\]
 
\end{lem}

\begin{proof}
The change in potential is given by 
\[
\potnew-\pot=\log\left[\frac{\int_{\kouthat\cap\hi}\exp(-t\cdot\inprod{\vc}{\vx})\ d\vx}{\int_{\kouthat}\exp(-t\cdot\inprod{\vc}{\vx})\ d\vx}\right].
\]
We now apply \cref{thm:GrunbaumGeneral} to the right hand side, with
the function $f(\vx) = \exp(-t\cdot\inprod{\vc}{\vx} - A(t\vc))$, where $A(\vtheta) = \log \left[\int_{\kouthat} \exp(-\inprod{\vtheta}{\vx}) d\vx\right]$. Noting that each halfspace $\hi$ passes directly
through $\vxosi$, where $\vxos$ is the centroid of $\kouthat$ with respect
to $f$ (by the definition of $\vxos$ in \cref{eq:xoutstar-mean}), \cref{rem:GrunbaumSimpleCase} applies and gives the claimed volume change. 
\end{proof}

To capture the change in potential due to the update in $t$, we recall the alternative perspective to the function $\log\left[\int_{\kouthat} \exp(-t \inprod{\vc}{\vx}) d\vx\right]$ given by \cref{def:EntropicBarr} and derive properties of self-concordant barriers. 
\begin{lem}
\label[lem]{lem:EtChange}
Consider a $\nu$-self-concordant barrier $\barr:\textrm{int}(\kcal)\rightarrow\R$
over the interior of a convex set $\kcal\subseteq\R^d$. Define 
\begin{equation}
\Et{\barr}t\defeq\min_{\vx}\left[t\cdot \inprod{\vc}{\vx}+\barr(\vx)\right] \text{ and }
\vxt\defeq\arg\min_{\vx}\left[t\cdot\inprod{\vc}{\vx}+\barr(\vx)\right].\label{eq:defEt}
\end{equation}
Then for $0\leq h\leq\frac{1}{\sqrt{\nu}}$, we have 
\[
\Et{\barr}t+th \cdot\inprod{\vxt}{\vc}\geq\Et{\barr}{t(1+h)}\geq\Et{\barr}t+ht\cdot\inprod{\vc}{\vxt}-h^{2}\nu.
\]
\end{lem} 
\begin{proof}
Note that here the first inequality is fairly generic and holds for
any function $\barr.$ 
By definition of $\Et{\barr}{t(1+h)}$ and $\Et{\barr}t$ in  \cref{eq:defEt} and using
the fact that the value on the right hand side of \cref{eq:lem_Et_1}
is smaller than the expression evaluated at a fixed $\vx=\vxt,$we
have 
\begin{align}
\Et{\barr}{t(1+h)} & =\min_{\vx}\left[t(1+h)\cdot\inprod{\vx}{\vc}+\barr(\vx)\right]\label{eq:lem_Et_1}\\
 & \leq t(1+h)\cdot\inprod{\vxt}{\vc}+\barr(\vxt)\nonumber \\
 & =\Et{\barr}t+th \cdot \inprod{\vxt}{\vc}.\nonumber 
\end{align}
 We now prove the second inequality of the lemma. This one specifically
uses the self-concordance of $\barr.$ Observe first, by definition,
\begin{equation}
\Et{\barr}t=-\barr^{\ast}(-t\vc).\label{eq:EtEqualsPhiAst}
\end{equation}
Since $\barr$ is a self-concordant barrier (and hence,
a self-concordant function), \cref{thm:ConjugateSC}
implies that $\barr^{\ast}$ is a self-concordant function as well.
Then, by applying \cref{thm:QuadApproxErr} to $\barr^{\ast}$ under
the assumption $\|-th\vc\|_{-t\vc}\leq1$ yields 
\begin{equation}
\barr^{\ast}(-t\vc-th\vc)\leq\barr^{\ast}(-t\vc)+\inprod{\nabla\barr^{\ast}(-t\vc)}{-th\vc}+\left[\frac{1}{2}\|-th\vc\|_{-t\vc}^{2}+\frac{\|-th\vc\|_{-t\vc}^{3}}{3(1-\|-th\vc\|_{-t\vc})}\right].\label[ineq]{eq:PhiAstQuadErr}
\end{equation}
By applying the first-order optimality condition to the definition
of $\vxt$ in \cref{eq:defEt}, we see that 
\begin{equation}
\nabla\barr(\vxt)=-t\vc.\label{eq:GradXt}
\end{equation}
 Next, evaluating $a\defeq\|-th\vc\|_{-t\vc}$ to check the assumption $\|-th\vc\|_{-t\vc}\leq1$, we get 
\begin{align*}
a^{2}=h^{2}\inprod{-t\vc}{\nabla^{2}\barr^{\ast}(-t\vc)\cdot(-t\vc)} & =h^{2}\inprod{\nabla\barr(\vxt)}{\nabla^{2}\barr^{\ast}(\nabla \barr(\vxt))\cdot\nabla\barr(\vxt)}\\
 & =h^{2}\inprod{\nabla\barr(\vxt)}{(\nabla^{2}\barr(\vxt))^{-1}\cdot\nabla\barr(\vxt)}\\
 & \leq h^{2}\nu
\end{align*}
where we used \cref{eq:GradXt} and \cref{lem:HessianConjugate},
in the first two equations and \cref{def:SelfConcordanceBarr} and
the complexity value of $\barr$ in the last step. Our range of $h$
proves that $a\leq1$, which is what we need for \cref{eq:PhiAstQuadErr}
to hold. We continue our computation to get 
\begin{align}
\left[\frac{1}{2}\|-th\vc\|_{-t\vc}^{2}+\frac{\|-th\vc\|_{-t\vc}^{3}}{3(1-\|-th\vc\|_{-t\vc})}\right]\leq\frac{1}{2}h^{2}\nu+\frac{1}{3}h^{3}\nu^{3/2} & \leq\frac{1}{2}h^{2}\nu+\frac{1}{3}h^{2}\nu\leq h^{2}\nu.\label[ineq]{eq:temp1}
\end{align}
Applying \cref{lem:GradConjugate} to \cref{eq:GradXt} gives
\begin{equation}
\nabla\barr^{\ast}(-t\vc)=\vxt.\label{eq:nablaPhiAstXt}
\end{equation} Plugging \cref{eq:nablaPhiAstXt} and \cref{eq:temp1} into the first
and second-order terms, respectively, of \cref{eq:PhiAstQuadErr} gives
\[
\barr^{\ast}(-t\vc-th\vc)\leq\barr^{\ast}(-t\vc)+\inprod{\vxt}{-th\vc}+h^{2}\nu.
\]
 Plugging in \cref{eq:EtEqualsPhiAst} gives the desired inequality
and completes the proof. 
\end{proof}

To finally compute the potential change due to $t,$ we need the following
result about the self-concordance parameter of the entropic barrier.
While \cite{bubeck2015entropic} prove that this barrier on a set in $\R^d$ is $(1+\epsilon_{d})d$-self-concordant,
the recent work of \cite{chewi2021entropic} remarkably improves this
complexity to exactly $d$. 
\begin{thm}[\cite{chewi2021entropic}]
\label{thm:chewi} The entropic barrier on any convex body $\kcal\subseteq\R^d$
is a $d$-self-concordant barrier. 
\end{thm}

We may now compute the potential change due to change in $t$ in \cref{line:updateT-Rule}. 
\begin{lem}\label[lem]{lem:PotChangeT}
When $t$ is updated to $t\cdot\left[1+\frac{\eta}{4 m}\right]$
in \cref{line:updateT-Rule} of \cref{alg:decomposable-fn-min}, the potential $\pot$ \cref{eq:TotalPotential} increases to $\potnew$ as follows: $$\potnew \leq \pot + \eta+\eta^{2}.$$ 
\end{lem}

\begin{proof}
Recall that the barrier function we use for the set $\kouthat$ is
the entropic barrier $\barrout.$ By \cref{eq:defEt} and the definition of conjugate,
we have 
\begin{align*}
-\Et{\barrout}t & =\max_{\vv}\left[\inprod{-t\vc}{\vv}-\barrout(\vv)\right]=\barrout^{\ast}(-t\vc).
\end{align*}
Applying \cref{def:EntropicBarr}, taking the conjugate on
both sides of the preceding equation, and using \cref{lem:fastast}
then gives 
\begin{equation}
-\Et{\barrout}t=\log\left[\int_{\kouthat}\exp(-t\cdot\inprod{\vc}{\vu})\ d\vu\right].\label{eq:PotTUpdate1}
\end{equation} From \cref{eq:TotalPotential},
the change in potential by changing $t$ to $t\cdot(1+h)$ for some $h>0$ may be expressed as 
\[
\potnew-\pot=\log\left[\int_{\kouthat}\exp\inprod{-t(1+h)\vc}{\vv}d\vv\right]-\log\left[\int_{\kouthat}\exp\inprod{-t\vc}{\vv}d\vv\right]+\inprod{th\cdot\vc}{\vx}.
\]
By applying $h=\frac{\eta}{4 m}$ and $\nu=m$ (via a direct application
of \cref{thm:chewi}), we have $h=\frac{\eta}{4m}\leq\frac{1}{\sqrt{m}}=\frac{1}{\sqrt{\nu}}$, and
so we may now apply \cref{eq:PotTUpdate1} and \cref{lem:EtChange}
in the preceding equation to obtain the following bound. 
\[
\potnew-\pot\leq th\inprod{\vc}{\vx}-th\inprod{\vc}{\vxt}+h^{2}\nu.
\]
From \cref{eq:xoutstar-min} and \cref{eq:defEt}, we see that $\vxt$ for the entropic
barrier satisfies the equation $\vxt=\vxos$, and applying the guarantee
$\inprod{\vc}{\vx}\leq\inprod{\vc}{\vxos}+\frac{4 m}{t}$ to this
inequality, we obtain 
\begin{align*}
\potnew-\pot & \leq th\cdot\frac{4 m}{t}+h^{2}\nu=\eta+\left(\frac{\eta}{4m}\right)^{2}\nu\leq\eta+\eta^{2}.
\end{align*}
\qedhere 
\end{proof}

\subsection{Potential change for the universal terms \label{sec:InnerPotChange}}

In this section, we study the change in volume on growing the inner
convex set $\kini$ in \cref{line:KinUpdated}. As mentioned in \cref{sec:AlgOverview},
our barrier of choice on this set is the universal barrier introduced in  \cite{nesterov1994interior}
(see also \cite{guler1997self}). 
This barrier was constructed to
demonstrate that \emph{any} convex body in $\R^n$ admits an $O(n)$-self-concordant
barrier, and its complexity parameter was improved to exactly $n$
in \cite{lee2021universal}. 

Conceptually, we choose the universal barrier for the inner set because the operation we perform on the inner set (i.e., generating its convex
hull with an external point $\vxos$) is dual to the operation of
intersecting the outer set with the separating halfspace containing
$\vxos$ (see \cref{lem:polarAfterAddingPoint}), which suggests the
use of a barrier dual to the entropic barrier used on the outer set.
As explained in \cite{bubeck2015entropic}, for the special case of
convex cones, the universal barrier is precisely one such barrier. 

We now state a technical property of the universal barrier, which
we use in the potential argument for this section. 
\begin{lem}[{\cite[Lemma 1]{lee2021universal},  \cite{nesterov1994interior,guler1997self}}]
\label[lem]{lem:UniversalBarrierCalculus}
Given a convex set $\kcal\in\R^{d}$ and $\vx\in\kcal$, let $\barrUniK(\vx)\defeq\log\vol(\kcal-\vx)^{\circ}$
be the universal barrier defined on $\kcal$ with respect to $\vx.$ 
Let $\mu\in\R^{d}$ be the center of gravity and $\Sigma\in\R^{d\times d}$
be the covariance matrix of the body $(\kcal-\vx)^{\circ}$, where $(\kcal-\vx)^\circ=\{\vy\in \R^n:\vy^\top(\vz-\vx)\leq 1,\forall \vz\in \kcal\}$ is the polar set of $\kcal$ with respect to $\vx$.
Then, we
have that 
\[
\nabla\barrUniK(\vx)=(d+1)\mu,\ \nabla^{2}\barrUniK(\vx)=(d+1)(d+2)\Sigma+(d+1)\mu\mu^{\top}.
\]
\end{lem}

\begin{lem}
\label[lem]{lem:UniversalBarrierPotentialChange} Given a convex set $\kcal\subseteq\R^d$
and a point $\vx\in\kcal.$ Let $\barrUniK\defeq\log\vol(\kcal-\vx)^{\circ}$
be the universal barrier defined on $\kcal$ with respect to $\vx.$
Let $\eta \leq 1/4$ and $\vy\in\kcal$ be a point satisfying the following condition
\begin{equation}
\inprod{\nabla\barrUniK(\vx)}{\vy-\vx}+\eta\|\vy-\vx\|_{\vx}\geq 4d,\label{eq:violatedDistCond}
\end{equation}
and construct the new set $\conv\left\{ \kcal,\vy\right\} .$ Then, the
value of the universal barrier defined on this new set with respect
to $\vx$ satisfies the following inequality. 
\[
\barrUniKNew(\vx)\defeq\psi_{\conv\left\{ \kcal,\vy\right\} }(\vx)=\log\vol(\conv( \kcal,\vy) -\vx)^{\circ}\leq\barrUniK(\vx)+\log(1-1/e+\eta).
\]
\end{lem}

\begin{proof}
By \cref{lem:polarAfterAddingPoint}, we have that 
\[
\left(\conv(\kcal,\vy) -\vx\right)^{\circ}\subseteq(\kcal-\vx)^{\circ}\cap\h,
\]
where $\h=\left\{ \vz\in\Rn:\inprod{\vz}{\vy-\vx}\leq1\right\} $.
Our strategy to computing the deviation of $\barrUniKNew(\vx)\defeq\barr_{\conv( \kcal,\vy) }(\vx)=\log\vol(\conv(\kcal,\vy) -\vx)^{\circ}$
from $\barrUniK(\vx)$ is to compute the change in $\vol(\conv(\kcal,\vy) -\vx)^{\circ}\leq\vol\left[(\kcal-\vx)^{\circ}\cap\h\right]$
from $\vol(\kcal-\vx)^{\circ},$ for which it is immediate that one may
apply an appropriate form of Gr\"unbaum's Theorem. 

Let $\mu$ be the center of gravity of the body $(\kcal-\vx)^{\circ}$.
If $\mu \notin \h$, then \cref{cor:GrunbaumVolApprox}
(with $\eta=0$) gives 
\[
\vol\left[(\kcal-\vx)^{\circ}\cap\h\right]\leq\vol(\kcal-\vx)^{\circ}\cdot(1-1/e),
\]
and taking the logarithm on both sides gives the claimed bound. We
now consider the case in which $\mu\in\h$, and the variance matrix
of the body $(\kcal-\vx)^{\circ}$ is $\Sigma$. Define $\vv = \vy - \vx$, and consider the point 
\[
\vz=\mu+\frac{1-\inprod{\vv}{\mu}}{\|\vv\|_{\Sigma}^{2}}\cdot\Sigma\vv.
\]
This point satisfies $\inprod{\vv}{\vz}=1$, which implies $\vz\in\h.$
Specifically, $\vz$ lies on the separating hyperplane. We show that
$\vz$ is sufficiently close to $\mu,$ so that even though $\mu\in\h,$
the subset of $(\kcal-\vx)^{\circ}$ cut out by the halfspace $\h$ is
not too large. By applying \cref{lem:UniversalBarrierCalculus} to
compute $\|\vv\|_{\vx}^{2}=(d+1)(d+2)\|\vv\|_{\Sigma}^{2}+(d+1)\inprod{\vv}{\mu}^{2}$,
we may compute the following quantity. 
\begin{align}
\|\vz-\mu\|_{\Sigma^{-1}} & =\frac{1-\inprod{\vv}{\mu}}{\sqrt{\frac{1}{(d+1)(d+2)}\|\vv\|_{\vx}^{2}-\frac{1}{d+2}\cdot\inprod{\vv}{\mu}^{2}}}  \nonumber \\
& =\sqrt{(d+1)(d+2)}\cdot\frac{1-\inprod{\vv}{\mu}}{\sqrt{\frac{1}{2}\|\vv\|_{\vx}^{2}+\frac{1}{2}\|\vv\|_{\vx}^{2}-(d+1)\inprod{\vv}{\mu}^{2}}}.\label{eq:uni-barr-pot-1}
\end{align}
Applying the expression for gradient from \cref{lem:UniversalBarrierCalculus}
in \cref{eq:violatedDistCond}, we have 
\begin{align}
\eta\|\vv\|_{\vx} & \geq 4 d-(d+1)\inprod{\vv}{\mu}\label{eq:uni-barr-pot-2}
  \geq 2d\inprod{\vv}{\mu},\nonumber 
\end{align}
where we used the fact that $\mu\in\h$ implies $\inprod{\vv}{\mu}\leq1.$
Since $\eta\leq1/4$, we have $ \frac{1}{2}\|\vv\|_{\vx}^{2}\geq(d+1)\inprod{\vv}{\mu}^{2}$. Plugging this in \cref{eq:uni-barr-pot-1} gives 
\begin{align*}
\|\vz-\mu\|_{\ma^{-1}} & \leq\sqrt{(d+1)(d+2)}\cdot\frac{1-\inprod{\vv}{\mu}}{\sqrt{\frac{1}{2}\|\vv\|_{\vx}^{2}}}
\leq 4d \frac{1-\inprod{\vv}{\mu}}{\|\vv\|_{\vx}} \\
& \leq 4 d \cdot\frac{1-\inprod{\vv}{\mu}}{4d (1-\inprod{\vv}{\mu}) / \eta}\leq\eta,
\end{align*}
which implies \cref{cor:GrunbaumVolApprox} applies,
giving us the desired volume reduction. 
\end{proof}

\subsection{Potential change for the update of \texorpdfstring{$\vx$}{x} \label{sec:IPM-analysis} }

In this section, we quantify the amount of progress made in \cref{line:IPM-step}
of \cref{alg:decomposable-fn-min} by computing the change in the potential
$\pot$ as defined in \cref{eq:TotalPotential}.
\begin{lem}\label[lem]{lem:potChangeX}
Consider the potential $\pot$ \cref{eq:TotalPotential} and the update step $\delta_{\vx}=\frac{\eta}{2}\cdot\frac{\vxos-\vx}{\|\vxos-\vx\|_{\vx,1}}$
as in \cref{line:IPM-step}. Assume the guarantees in \cref{eq:Condition2}
and \cref{eq:Condition1}. Then the potential $\pot$ incurs the following
minimum decrease. 
\[
\potnew\leq\pot-\frac{\eta^{2}}{4}.
\]
\end{lem}

\begin{proof}
Taking the gradient of $\pot$ with respect to $\vx$ and rearranging
the terms gives
\begin{equation}
t\vc=\nabla_{\vx}\pot-\sum_{i=1}^{n}\nabla\barrini(\vxi).\label{eq:potential-ipm-3}
\end{equation}
By applying the expression for $t\vc$ from the preceding equation,
we get 
\begin{align}
\potnew-\pot & =t\inprod{\vc}{\vx+\delta_{\vx}}+\sum_{i=1}^{n}\barrini(\vxi+\delta_{\vx,i})-t\inprod{\vc}{\vx}-\sum_{i=1}^{n}\barrini(\vxi)\nonumber \\
 & =\langle\nabla_{\vx}\pot,\delta_{\vx}\rangle+\sum_{i=1}^{n}\underbrace{\left[\barrini(\vxi+\deltaXi)-\barrini(\vxi)-\inprod{\nabla\barrini(\vxi)}{\deltaXi}\right]}_{q_{\barrini}(\vxi)}.\label{eq:potential-ipm-2}
\end{align}
The term $q_{\barrini}(\vxi)$ measures the error due to first-order
approximation of $\barrini$ around $\vxi$. Since $\barrini(\vxi)$ is self-concordant functions and $\|\delta_{\vx,i}\|_{\vx_{i}}\leq\|\delta_{\vx}\|_{\vx,1}\leq\eta\leq1/4$, 
\cref{thm:QuadApproxErr} shows that

\begin{equation}
\barrini(\vxi+\deltaXi)-\barrini(\vxi)-\inprod{\nabla\barrini(\vxi)}{\deltaXi}\leq\|\delta_{\vx,i}\|_{\vx,i}^{2}.\label[ineq]{eq:potential-ipm-1}
\end{equation}
Plugging in \cref{eq:potential-ipm-1}
into \cref{eq:potential-ipm-2}, we get 
\begin{equation}
\potnew-\pot \leq\langle\nabla_{\vx}\pot,\delta_{\vx}\rangle+\|\delta_{\vx}\|_{\vx,1}^{2}.\label[ineq]{eq:potential-ipm-4}
\end{equation}
We now bound the two terms on the right hand side one at a time. Using
the definition of $\deltaX$ (as given in the statement of the lemma)
and of $\nabla_{\vx}\pot$ from \cref{eq:potential-ipm-3} gives
\begin{align}
\langle\nabla_{\vx}\pot,\delta_{\vx}\rangle & =\frac{\eta}{2}\frac{1}{\|\vxos-\vx\|_{\vx,1}}\langle\nabla_{\vx}\pot,\vxos-\vx\rangle\nonumber \\
 & =\frac{\eta}{2}\frac{1}{\|\vxos-\vx\|_{\vx,1}}\left[\inprod{t\vc}{\vxos-\vx}+\sum_{i=1}^{n}\inprod{\nabla\barrini(\vxi)}{\vx_{\textrm{out},i}^{\star}-\vxi}\right]\nonumber \\
 & \leq\frac{\eta}{2}\frac{1}{\|\vxos-\vx\|_{\vx,1}}\left[\inprod{t\vc}{\vxos-\vx}+\sum_{i=1}^{n}\left(4 r_{i}-\eta\|\vxosi-\vxi\|_{\vx_{i}}\right)\right]\nonumber \\
 & =\frac{\eta}{2}\frac{1}{\|\vxos-\vx\|_{\vx,1}}\left[\inprod{t\vc}{\vxos-\vx}+4m-\eta\|\vxos-\vx\|_{\vx,1}\right]\nonumber \\
 & \leq\frac{\eta}{2}\frac{1}{\|\vxos-\vx\|_{\vx,1}}\cdot\left(-\eta\|\vxos-\vx\|_{\vx,1}\right)\nonumber \\
 & =-\eta^{2}/2.\label[ineq]{eq:pot-ipm-5}
\end{align}
where the third step follows from \cref{eq:Condition2}, the fourth
step follows from $\sum_{i=1}^{n}d_{i}=m$, and the fifth step follows
from \cref{eq:Condition1}. To bound the second term, we note from
\cref{line:IPM-step} that 
\begin{equation}
\|\delta_{\vx}\|_{\vx,1}^{2}=\left(\frac{\eta}{2}\cdot\frac{\|\vxos-\vx\|_{\vx,1}}{\|\vxos-\vx\|_{\vx,1}}\right)^{2}=\eta^{2}/4.\label{eq:pot-ipm-6}
\end{equation}
Hence, we may plug in \cref{eq:pot-ipm-5} and \cref{eq:pot-ipm-6}
into \cref{eq:potential-ipm-4} to get the desired result. 
\end{proof}

\subsection{Total oracle complexity}\label{sec:total-complexity}
Before we bound the total oracle complexity of the algorithm, we first bound the total potential change throughout the algorithm.

\begin{lem}\label{lem:potential-change}
Consider the potential function $\pot=t\inprod{\vc}{\vx}+\log\left[\int_{\kouthat}\exp(-t\inprod{\vc}{\vu})d\vu\right]+\sum_{i\in[n]}\barrini(\vxi)$ as defined in \cref{eq:TotalPotential} associated with \cref{alg:decomposable-fn-min}. Let $\potinit$ be the potential at $t=\tinit$ of this algorithm, and let $\potend$ be the potential at $t = \tend$. Suppose at $t=\tinit$ in \cref{alg:decomposable-fn-min}, we have $\ball_m(\vx,\bar r)\subseteq \kin$ with $\bar r = r / \operatorname{poly}(m)$ and  $\kout\subseteq\ball_m(0,\bar R)$ with $\bar R = O(\sqrt nR)$. Then we have, under the assumptions of \cref{thm:MainThmOfKiProblem}, that \[
\potinit - \potend \leq O\left(m\log\left(\frac{m R}{\epsilon r}\right)\right).
\]
\end{lem} 
\begin{proof} For this proof, we introduce the following notation: let $\vol_{\ma}(\cdot)$ denote the volume restricted to the subspace $\{\vx:\ma\vx=\vb\}$. We also invoke \cref{fact:volSphereCone}. We now bound the change in the potential term by term, starting with the entropic terms 
\[ t\inprod{\vc}{\vx} + \log\left[\int_{\kouthat}\exp(-t\inprod{\vc}{\vu})d\vu\right]\numberthis\label{eq:RemainingTermPotential}\] 
at $t= \tinit$ and a lower bound on it at $t= \tend$. We start with bounding \cref{eq:RemainingTermPotential} evaluated at $t= \tend = \frac{8m}{\epsilon\|\vc\|_2R}$. 

Let $\bar \vx=\arg\min_{\vx\in \kouthat} \inprod{\vc}{\vx}$ and $\alpha = \langle \vc,  \bar \vx\rangle$. By optimality of $\bar \vx$, we know that $\bar\vx \in \partial \kouthat$. Denote $\ball_A(\vz,\bar r)$ to be $\ball(\vz,\bar r)$ restricted to the subspace $\{\vx:\ma\vx=\vb\}$. Note that $\kouthat\supseteq\ball_{\ma}(\vz,\bar r)$. Consider the cone $\mathcal{C}$ and halfspace $\mathcal{H}$ defined by \[\mathcal{C} = \bar \vx + \left\{\lambda \vy: \lambda>0, \vy \in  \ball_{\ma}(\vz-\bar \vx, \bar r)  \right\}\text{ and }
    \mathcal{H}\defeq \left\{\vx:\langle \vc, \vx\rangle \leq \alpha +\frac{1}{\tend}\right\}.
\]
Then, by a similarity argument, we note that $\mathcal{C}\cap \mathcal{H}$ contains a cone with height $\frac{1}{\tend\|c\|_2}$ and base radius $\frac{\bar r}{\bar R\tend\|c\|_2}$, which means
\begin{align*}
    \vol_{\ma}(\mathcal{C}\cap \mathcal{H}) &\geq \frac{1}{m - \operatorname{rank}(\ma)}\cdot\frac{1}{\tend\|c\|_2}\cdot  \left(\frac{\bar r}{\bar R\tend\|c\|_2}\right)^{m-\operatorname{rank}(\ma)-1} \cdot \vol(\ball_{m-\operatorname{rank}(\ma)-1}(0,1)).
    \end{align*}
Then, we have
\begin{align*}
 \log\left[\int_{\kouthat}\exp(-\tend \inprod{\vc}{\vu})d\vu\right] + \tend \inprod{\vc}{\vx}
    &\geq \log \left[\int_{\kouthat} \exp(-\tend \inprod{\vc}{\vu}) d\vu \right] + \tend \min_{\vx\in \kouthat} \inprod{\vc}{\vx}  \\
    &\geq \log \left[\int_{\mathcal{C}\cap \mathcal{H}} \exp(-\tend \inprod{\vc}{\vu}) d\vu \right] + \tend\alpha \\
    &\geq \log \left[\int_{\mathcal{C}\cap \mathcal{H}} \exp(-\tend\alpha -1) d\vu \right] + \tend \alpha \\ 
    &=\log\left[\frac{1}{e}\cdot \vol_{\ma}(\mathcal{C}\cap \mathcal{H}) \exp(-\tend\alpha ) \right] + \tend \alpha \\ 
    &= \log\left[\vol_{\ma}(\mathcal{C}\cap \mathcal{H})\cdot \frac{1}{e}\right] \\
    &\geq - (m-\operatorname{rank}(\ma)-1) \cdot\log({\bar R\tend \|c\|_2}/{\bar r}))\\
    &\quad +  \log(\vol(\ball_{m-\operatorname{rank}(\ma)-1}(0,1))) \\
    &\quad - \log (m - \operatorname{rank}(\ma)) - \log (\tend\|\vc\|_2) -1. \numberthis\label[ineq]{eq:PhiEnd1}
    \\    \end{align*}

Next, to bound \cref{eq:RemainingTermPotential} at $t=\tinit$, we may express these terms as follows. 
\begin{align*}
   \log\left[\int_{\kouthat}\exp(-\tinit \cdot \inprod{\vc}{\vu})d\vu\right] +  \tinit \cdot \inprod{\vc}{\vx}
    \\    \\    &\leq\log\left[\vol_{\ma}(\kouthat) \right]+\tinit\cdot\max_{\vu\in\kouthat}\left\langle \vc,\vx-\vu\right\rangle\\
    &\leq\log(\vol(\ball_{m-\operatorname{rank}(\ma)}(0, \bar R)))+\tinit\cdot2\bar R\|\vc\|_{2}\\
    &\leq \log(\vol(\ball_{m-\operatorname{rank}(\ma)}(0, 1))) \\
        & \quad +(m-\operatorname{rank}(\ma)) \log{\bar R} + O(m\log{m}), \numberthis\label[ineq]{eq:PhiInit1}
\end{align*}
where  the second step is by $\kouthat\subseteq \kout\subseteq \ball_{\sum_{i\in[n]} d_i}(0, \bar R)$ (here, the second inclusion is by assumption), and the third step is by $\vol(\ball_{q}(0, \bar R))= \frac{\pi^{q/2}}{\Gamma(1+ q/2)} {\bar R}^q$ and our choice of  $\tinit \defeq  \frac{m\log{m}}{\sqrt{n} \|\vc\|_2 R}.$

We now compute the change in the entropic barrier $\sum_{i \in [n]} \barrini(\vxi)$, where \[\barrini(\vxi) = \log \vol(\kini^\circ(\vxi)).\] Define $\ball_{d}(0, r)$ to be the $d$-dimensional Euclidean ball centred at the origin and with radius $r$. We note by the radius assumption of \cref{thm:MainThmOfKiProblem} that  $\kini\subseteq\ki \subseteq \ball_{d_i}(0, \bar R)$ throughout the algorithm. By the assumption made in this lemma's statement, we have that at the start of \cref{alg:decomposable-fn-min},  $\kini\supseteq\ball_{d_i}(\vx, \bar r)$. These give us the following bounds.   
\[
\barrini^{\text{end}}(\vxi)\geq \log(\vol(\ball^\circ_{d_i}(0,\bar R)) \text{ and }  \barrini^{\text{init}}(\vxi) \leq  \log(\vol(\ball_{d_i}^{\circ}(\vxi, \bar r))). 
\]
Applying the fact that $\vol(\ball_d(0, r)) \propto r^d$ and summing over all $i \in [n]$ gives 
\begin{align*}
    \sum_{i \in [n]} \left[\barrini^{\text{init}}(\vxi) -\barrini^{\text{end}}(\vxi)\right] 
   \\    &\leq \sum_{i \in [n]}\log\left(\frac{\vol(\ball_{d_i}(\vxi, 1/\bar r))}{\vol(\ball_{d_i}(0, 1/\bar R))}\right) \\
    &=\sum_{i \in [n]} d_i\log(\bar R/\bar r) = m \log (\bar R/\bar r). \numberthis\label[ineq]{eq:PsiChangeInitEnd}
\end{align*}

Combining \cref{eq:PhiInit1}, \cref{eq:PhiEnd1}, and \cref{eq:PsiChangeInitEnd}, we have 
\begin{align*}
\potinit - \potend &\leq m\log(mR/r) \\
&\quad +\left[\log(\vol(\ball_{m-\operatorname{rank}(\ma)}(0, 1)))
        +(m-\operatorname{rank}(\ma)) \log{\bar R} + O(m\log{m})\right]\\
& \quad+ (m-\operatorname{rank}(\ma)-1) \cdot\log({\bar R\tend \|c\|_2}/{\bar r}) - \log(\vol(\ball_{m-\operatorname{rank}(\ma)-1}(0,1))) \\
& \quad + \log (m - \operatorname{rank}(\ma)) + \log (\tend\|\vc\|_2) + 1 \\
&\leq m\log(mR/\epsilon r) \\
& \quad + O(m \log m) \\ 
&\quad + O((m - \operatorname{rank}(\ma))\log(mR/\epsilon r)) \leq O(m \log (mR/\epsilon r)). 
\end{align*}
\qedhere 
\end{proof}

\begin{restatable}{lem}{lemAlgOracleComplexity}[Total oracle complexity]\label[lem]{lem:totalOracleComplexity} Suppose the inputs $\kin$ and $\kout$ to \cref{alg:decomposable-fn-min} satisfy $\kout\subseteq\ball_m(0,\bar{R})$ with $\bar R = O(\sqrt nR)$ and $\kin \supseteq \ball(\vz, \bar r)$ with $\bar r = r / \operatorname{poly}(m)$. Then, when \cref{alg:decomposable-fn-min} terminates at $t \geq \tend$, it outputs a solution $\vx$ that satisfies 
\[
\vc^\top \vx \leq \min_{\vx\in\kcal,A\vx=\vb} \vc^\top\vx+  \epsilon\cdot \|\vc\|_2R
\]
using at most $\nsep = O\left( m \log\left( \frac{mR}{\epsilon r}\right)\right)$ separation oracle calls.
\end{restatable} 
\begin{proof}
 Let $\nt$ be the number of times $t$ is updated; $\nin$ the number of times $\kin$ is updated; $\nout$ the number of times $\kout$ is updated; $\nx$ the number of times $\vx$ is updated, and $\nn$ the total number of iterations of the \texttt{while} loop before termination of \cref{alg:decomposable-fn-min}. 
 Then, combining \cref{lem:PotKOutChange}, \cref{lem:PotChangeT}, \cref{lem:UniversalBarrierPotentialChange}, and \cref{lem:potChangeX} gives
\[ \potend\leq \potinit + \nout \cdot \log(1-1/e) + \nt\cdot (\eta + \eta^2) + \nin \cdot \log(1-1/e + \eta) + \nx\cdot \left(-\frac{\eta^2}{4}\right).\numberthis\label[ineq]{eq:totalPotChange}\] 
The initialization step of \cref{alg:decomposable-fn-min} chooses $\eta = 1/100$, $\tend = \frac{8 m }{\epsilon \|\vc\|_2 R}$, and $\tinit = \frac{m\log(m)}{\sqrt{n} \|\vc\|_2 R}$, and we always update $t$ by a multiplicative factor of $1+\frac{\eta}{4 m}$ (see \cref{line:updateT-Rule}); therefore, we have 
\[
\nt = O(m\log(mR/(\epsilon r)).
\] 
From \cref{alg:decomposable-fn-min}, the only times the separation oracle is invoked is when updating $\kin$ or $\kout$ in \cref{line:KinUpdated} and \cref{line:KoutUpdated}, respectively.  Therefore, the total separation oracle complexity is $\nsep = \nin + \nout$. Therefore, we have \[
 \nsep = \nin + \nout \leq O(1)\cdot \left[\potinit - \potend +\nt \right] =  O(m\log(mR/(\epsilon r))
 \]  This gives the claimed separation oracle complexity. 
 
 We now prove the guarantee on approximation. Let $\vx_{\text{output}}$ be the output of \cref{alg:decomposable-fn-min} and $\vx$ be the point which entered \cref{line:updateT} right before termination. Note that the termination of \cref{alg:decomposable-fn-min} implies, by  \cref{line:updateT}, that \[
 \vc^\top \vx_{\text{output}} \leq \vc^\top \vx + \frac{\nu}{\tend}\leq \vc^\top \vxos + \frac{4(n+m)}{\tend}  \leq \min_{\vx\in\kcal,A\vx=\vb} \vc^\top\vx+  \epsilon\cdot \|\vc\|_2 \cdot R
 \]
 where the first step is by the second inequality in \cref{lem:two-sided-ineq} (using the universal barrier) and the last step follows by our choice of $\tend$ and the definition of $\vxos$ and $\kout\supseteq \kcal$.
\end{proof}

\begin{thm}[Main theorem of \cref{eq:1main}]\label{thm:MainThmOfKiProblem}
Given the convex program
\[
\begin{array}{ll}
\mbox{minimize} & \inprod{\vc}{\vx}, \\
\mbox{subject to } & \vxi\in\ki\subseteq \R^{d_i + 1}\forall i\in[n],\\
& \ma \vx = \vb. 
\end{array}
\]
Denote $\kcal = \kcal_1 \times \kcal_2 \times \dotsc \times \kcal_n$. Assuming we have
\begin{itemize}
    \item outer radius $R$: For any $\vx_i\in \ki$, we have $\|\vx_i\|_2 \leq R$, and
    \item inner radius $r$: There exists a $\vz \in \R^d$ such that $\ma\vz=\vb$ and $\ball(\vz,r)\subset \kcal$,
\end{itemize}
then, for any $0<\epsilon<\frac12$, we can find a point $\vx \in \kcal$ satisfying $\ma \vx = \vb$ and
\[
\inprod{\vc}{\vx} \leq \min_{\substack{\vxi\in\ki\subseteq \R^{d_i + 1}\forall i\in[n],\\ \ma\vx=\vb}}\inprod{\vc}{\vx} + \epsilon \cdot \|\vc\|_2 \cdot R,
\] 
in $O(\poly(m \log(mR/\epsilon r)))$ time and using 
\[
O(m\log(mR/(\epsilon r))
\] gradient oracle calls, where $m = \sum_{i=1}^n d_i$.
\end{thm}
\begin{proof}
We apply \cref{thm:initOne} for each $\ki$ separately to find a solution $\vz_i$. Then $\vz =(\vz_1, \dots, \vz_n) \in \R^{m+n}$ satisfies $\ball_{m+n}(\vz,\bar r)\subset \kcal$ with $\bar r=\frac{r}{6m^{3.5}}$. 
Then, we modified convex problem as in \cref{def:ModConvProg} with $s=2^{16}\frac{m^{2.5} R}{r\epsilon}$ and obtaining the following:
\[
\begin{array}{ll}
    \mbox{minimize} & \inprod{\bar \vc}{\bar \vx} \\ 
    \mbox{subject to} & \bar\ma \bar \vx = \bar \vb,\\
    &\bar \vx \in \bar \kcal\defeq\kcal \times \R_{\geq 0}^{m+n}\times\R_{\geq 0}^{m+n}\numberthis \label[prob]{eq:modifiedConvexProgram}
\end{array}
\]
with 
\[
\bar \ma = [\ma \mid \ma \mid -\ma], \bar \vb = \vb, \bar \vc = (\vc, \frac{\|\vc\|_2 s}{\sqrt{m+n}}\cdot \mathbf{1},\frac{\|\vc\|_2 s}{\sqrt{m+n}}\cdot \mathbf{1})^\top
\]
We solve the linear system $\ma \vy=\vb-\ma\vz$ for $\vy$. Then, we construct the initial $\ovx$ by set $\ovx^{(1)}=\vz$, 
\[
\ovx^{(2)}_i = \begin{cases}
\vy_i &\text{if } \vy_i\geq 0,\\ 
0 &\text{otherwise.}
\end{cases}
\quad \text{ and } \quad 
\ovx^{(3)}_i = \begin{cases}
-\vy_i &\text{if } \vy_i< 0,\\ 
0 &\text{otherwise}.
\end{cases}
\]

Then, we run \cref{alg:decomposable-fn-min} on the \cref{eq:modifiedConvexProgram}, with 
initial $\ovx$ set above,
$\bar m = 3(m+n), 
\bar n = n +2, 
\bar \epsilon = \frac{\epsilon }{6\sqrt{n}s}, \overline{\mathcal{K}}_{\textrm{in}} =\{\vx^{(1)}\in B(\vz,\bar{r}),(\vx^{(2)},\vx^{(3)})\in\R_{\geq0}^{2n}\}$ and $ \kouthat = \ball_{\bar m}(\mathbf{0}, \sqrt{n}R)$. 

By our choice of $\tend$, we have \[
\bar{t}_{\textrm{end}} 
= \frac{8 \bar m}{\bar \epsilon \|\bar \vc\|_2\bar R } 
\leq\frac{48m}{\epsilon\|\vc\|_2 R}.
\]

First, we check the condition that $s\geq 48\bar \nu \bar{t}_{\textrm{end}} \sqrt{m+n} \frac{R^2}{r}\|\vc\|_2$, we note that
\[
48\bar \nu \bar{t}_{\textrm{end}} \sqrt{m+n} \frac{R^2}{r}\|\vc\|_2 
\leq 27648\frac{m^{2.5} R}{\epsilon r} \leq  2^{16}\frac{m^{2.5} R}{r\epsilon} = s.
\]

Let $\bar{\vx}_{output}=(\vx_{output}^{(1)},\vx_{output}^{(2)},\vx_{output}^{(3)})$ be the output of \cref{alg:decomposable-fn-min}. 
Then, let $\vx_{output}=\vx_{output}^{(1)}+\vx_{output}^{(2)}-\vx_{output}^{(3)}$ as defined in \cref{thm:initial-main}.
By \cref{lem:totalOracleComplexity}, we have 
\[
\min_{\vx\in \pin} \bar\vc^\top\ovx \leq \min_{\vx\in \mathcal{P}} \vc^\top\ovx +\gamma
\]
where $\gamma = \bar \epsilon \cdot \|\bar \vc\|_2 \cdot \bar R$.

Applying (3) of \cref{thm:initial-main}, we have 
\[
\vc^\top \vx_{output} \leq \frac{\bar \nu+1}{\bar{t}_{\textrm{end}}}+\gamma+\min_{x\in \kcal,A\vx=\vb} \vc^\top \vx \leq  \min_{x\in \kcal,A\vx=\vb} \vc^\top \vx + \epsilon\cdot \|\vc\|_2 \cdot R.
\]
The last inequality follows by our choice of $\bar \epsilon$ and $\bar{t}_{\textrm{end}}$, we have 
$\gamma \leq \frac{\epsilon}{2}\|\vc\|_2R$ and $\frac{\bar \nu+1}{\bar{t}_{\textrm{end}}}\leq \frac{\epsilon}{2}\|\vc\|_2R$.
Plug this $\bar\epsilon$ in \cref{lem:totalOracleComplexity}, it gives the claimed oracle complexity.

\qedhere
\end{proof}

\begin{thm}[Main Result]\label{thm:mainResult} Given \cref{prob:OurProblemStatement} and $\theta^{(0)}$ such that
$\|\theta^\star - \theta^{(0)}\|_2\leq R$. Assuming all the $f_i$'s are $L$-Lipschitz, then there is an algorithm that in time $\operatorname{poly}(m \log(1/\epsilon))$, using $O(m\log(m/\epsilon))$ gradient oracle calls, outputs a vector $\vtheta\in \R^d$ such that 
\[
\sum_{i=1}^n f_i(\vtheta) \leq \sum_{i=1}^n f_i(\theta^\star)  + \epsilon \cdot L R.
\]
\end{thm}

\begin{proof}
First, we reformulate \cref{eq:generalObj} using a change of variables and the epigraph trick. 
Suppose each $f_i$ depends on $d_i$ coordinates of $\vtheta$ given by $\{i_1, \dots, i_{d_i}\} \subseteq [d]$. Then,
symbolically define $\vx_i = [x^{(i)}_{i_1}; x^{(i)}_{i_2}; \dots; x^{(i)}_{i_{d_i}}] \in \R^{d_i}$ for each $i \in [n]$.
Since each $f_i$ is convex and supported on $d_i$ variables, its epigraph is convex and $d_i+1$ dimensional. So we may define the convex set
\[
\ki^{\textrm{unbounded}} = \left\{(\vx_i, z_i) \in \R^{d_i + 1}: f_i (\vx_i) \leq L z_i \right\}.
\]
Finally, we add linear constraints of the form $x^{(i)}_{k} = x^{(j)}_{k}$ for all $i,j,k$ where $f_i$ and $f_j$ both depend on $\vtheta_k$. We denote these by the matrix constraint $\ma \vx = \vb$.
Then, \cref{prob:OurProblemStatement} is equivalent to
\begin{equation} \label[none]{eq:linear-formulation}
\begin{array}{ll}
    \mbox{minimize}  & \sum_{i=1}^n L z_i \\
    \mbox{subject to} &  \ma \vx =\vb \\
      & (\vx_i, z_i) \in \ki^{\textrm{unbounded}} \text{ for each $i \in [n]$}.
\end{array}
\end{equation}

Since we are given $\vtheta^{(0)}$ satisfying $\norm{\vtheta^{(0)} - \vtheta^*}_2 \leq R$, 
we define $\vx_i^{(0)} = [ \vtheta^{(0)}_{i_1}; \dots, \vtheta^{(0)}_{i_{d_i}}]$ and 
$z_i^{(0)} = f_i(\vtheta^{(0)})/L$.
Then, we can restrict the search space $\ki^{\textrm{unbounded}}$ to
\begin{align*}
\ki &= \ki^{\textrm{unbounded}}  \cap \{ (\vx_i, z_i) \in \R^{d_i+1}: \|\vx_i-\vx_i^{(0)}\|_2 \leq R \text{ and }
    z_i^{(0)} -2R \leq z_i \leq z_i^{(0)} + 2R \}.
\end{align*}

It's easy to check that $\kcal_i$ is contained in a ball of radius $5R$ centered at $(\vx^{(0)}_i, z_i^{(0)})$, 
and contains a ball of radius $R$ centered at $(\vx^{(0)}_i, z_i^{(0)})$. 
The subgradient oracle for $f_i$ translates to a separation oracle for $\ki$.
Then, we apply \cref{thm:MainThmOfKiProblem} to \cref{eq:linear-formulation} with $\ki^{\textrm{unbounded}}$ replaced by $\ki$ to get the error guarantee and oracle complexity directly.
\end{proof}

Finally, we have the matching lower bound.

\thmLowerBound*
\begin{proof}
    \cite{DBLP:books/sp/Nesterov04} shows that for any $d_i$, there exists $f_i : \R^{d_i} \mapsto \R$ for which $\Omega(d_i \log (1/\epsilon))$ total gradient queries are required. 
    We define $f_1, \dots, f_n$ to be such functions on disjoint coordinates of $\vtheta$. 
    It follows that $\Omega (\sum_{i=1}^n d_i \log (1/\epsilon)) = \Omega (m \log (1/\epsilon))$ gradient queries are required in total.
\end{proof}

\section{Initialization}\label{sec:init}
\subsection{Constructing an initial \texorpdfstring{$\protect\kini$}{Kin,i}}\label{sec:initial-ki}

In this section, we discuss how to construct an initial set $\kini$ to serve as an input to \cref{alg:decomposable-fn-min}. In
particular, we will prove the following theorem. 
\begin{thm}\label{thm:initOne}
Suppose we are given separation oracle access to a convex set $\kcal$ 
that satisfies $\ball(\vz,r)\subseteq\kcal\subseteq \ball(\mathbf{0},R)$ for some $\vz\in\R^{d}$.
Then, \cref{alg:InnerBallFinding}, in $O(d\log(R/r))$ separation oracle calls to $\kcal$, outputs a point
$\vx$ such that $\ball\left(\vx,\frac{r}{6 d^{3.5}}\right)\subseteq \kcal$. 
\end{thm}

\begin{algorithm}[H]
\caption{Inner Ball Finding}
\label{alg:InnerBallFinding}

\begin{algorithmic}[1]

\algnewcommand{\LeftComment}[1]{\State \(\triangleright\) #1} 

\State $\kout\leftarrow B(0,R)$

\While{ \textbf{true} }
	\State Let $\vv$ be the center of gravity of $\kout$
	
	\State Sample $\vu$ from $B(\vv,r/(6d))$ uniformly
	
	\If{$\vu\in\mathcal{K}$}
		\State Let $S=\{\vv\pm\frac{r}{6d^3}\ve_{i}:i\in[d]\}$
		
		\If{ $S \subset \mathcal{K}$}
			\State \Return the inscribed ball of $\conv(S)$
		\EndIf
	\EndIf
	
	\State Let $\kout\leftarrow\kout\cap\mathcal{H}$ where $\mathcal{H}=\mathcal{O}(\vu)$

\EndWhile

\end{algorithmic}
\end{algorithm}

Before we prove the preceding theorem, we need the following facts about
the self-concordant barrier and convex sets. 

\begin{thm}[{\cite[Theorem 4.2.6]{DBLP:books/sp/Nesterov04}}]\label{thm:SCFactInit}
Let $\barr:\textrm{int}(\kcal)\rightarrow\R$ be a $\nu$-self-concordant barrier with the minimizer $\vx_{\barr}^{\star}$.
Then, for any $\vx\in\textrm{int}(\kcal)$ we have:

\[
\|\vx_{\barr}^{\star}-\vx\|_{\vx_{\barr}^{\star}}\leq\nu+2\sqrt{\nu}.
\]
On the other hand, for any $\vx\in\R^{d}$ such that $\|\vx-\vx_{\barr}^{\star}\|_{\vx_{\barr}^{\star}}\leq1$,
we have $x\in\textrm{int}(\kcal)$.
\end{thm}

\begin{thm}[{\cite[Theorem 4.1]{kannan1995isoperimetric}}]
\label{thm:circumscribed-ellipsoid}Let $\mathcal{K}\subseteq \R^d$ be a convex
set with center of gravity $\mu$ and covariance matrix
$\Sigma$. Then,
\[
\{\vx:\|\vx-\mu\|_{\Sigma^{-1}}\leq\sqrt{(d+2)/d}\}\subseteq\mathcal{K}\subseteq\{\vx:\|\vx-\mu\|_{\Sigma^{-1}}\leq\sqrt{d(d+2)}\}.
\]
\end{thm}

\begin{thm}[{\cite[Section 1.4.2]{brazitikos2014geometry}}]\label{thm:boundary_set}
Let $\mathcal{K}$ be a convex set with $\mathcal{K}\subset B(\vu,R)$
for some $R$. Let $\mathcal{K}_{-\delta}=\{\vx:B(\vx,\delta)\subset\mathcal{K}\}$.
Then, we have
\[
\vol\mathcal{K}_{-\delta}\geq \vol \mathcal{K}-(1-(1-\frac{\delta}{R})^{d})\cdot\vol B(\vu,R)
\]
\end{thm}

\begin{proof}[Proof of \cref{thm:initOne}]
 We note that by the description of the \cref{alg:InnerBallFinding}, the returned
ball is the inscribed ball of $\conv(S)$ and we have $\vv\in\mathcal{K}$
for each $\vv\in S$. Then, we must have $\conv(S)\subseteq\mathcal{K}$.
We note that $\conv(S)$ is a $\ell_1$ ball with $\ell_1$ radius $\frac{r}{6 d^{3}}$,
then the inscribed ball has $\ell_2$ radius $\frac{r}{6 d^{3.5}}.$

First, we prove the sample complexity of the algorithm above. We use
$\mathcal{K}_{t}$ to denote the $\kout$ at the $t$-th iteration.
We first observe that throughout the algorithm, $\mathcal{K}_{t}$
is obtained by intersection of halfspaces and $B(0,R)$. This implies
\[
B(\vz,r)\subseteq\mathcal{K}\subseteq\mathcal{K}_{t}\qquad\forall t.
\]
Since $\mathcal{K}_{t}$ contains a ball of radius $r$, let $A_{t}$
be the covariance matrix of $\mathcal{K}_{t}$. By \cref{thm:circumscribed-ellipsoid},
we have 
\[
A_{t}\succeq \frac{r^2}{d(d+2)} I.
\]
Let $\mathcal{H}_{t}$ be the halfspace returned by the oracle at iteration $t$. 
We note that $\vu$ is sampled uniform from $B(\vv,r/(6d))$, so we have
\[
\|\vv-\vu\|_{A^{-1}}\leq\frac{\sqrt{d (d+2)}}{r}\cdot\frac{r}{6d}\leq\frac{1}{3}.
\]
Apply the inequality above to \cref{cor:GrunbaumVolApprox}, we have
\[
\vol(\mathcal{K}_{t})\leq(1-1/e+1/3)^{t}\vol(\mathcal{K}_{0})\leq(1-1/30)^{t}\vol(B(0,R)).
\]
Then, since $B(\vz,r)\subseteq\mathcal{K}_{t}$ for all the $t$,
this implies the algorithm at most takes $O(d\log(R/r))$ many iterations. 

Now, we consider the number of oracle calls within each iterations. There are three possible cases to consider:
\begin{enumerate}
\item $\vu\in\mathcal{K}_{-\delta}$ with $\delta=\frac{r}{6d^{3}}$
(see the definition of $\mathcal{K}_{-\delta}$ in \cref{thm:boundary_set}). In this
case, we have $S\subset\mathcal{K}$ and this is the last iteration.
We can pay this $O(d)$ oracle calls for the last iteration.

\item $\vu\in\mathcal{K}\backslash\mathcal{K}_{-\delta}$.

Since $\vu$ is uniformly sampled from $B(\vv,r/(6d))$, \cref{thm:boundary_set} shows that
$\vu\in\mathcal{K}\backslash\mathcal{K}_{-\delta}$ with probability
at most
\[
1-(1-\frac{\delta}{r/(6d)})^{d}\leq\frac{1}{d}.
\]
Hence, this case only happens with probability only at most $1/d$.
Since the cost of checking $S\subset\mathcal{K}$ takes $O(d)$ oracle
calls. The expected calls for this case is only $O(1)$.

\item $\vu\notin K$. The cost is just $1$ call.
\end{enumerate}
Combining all the cases, the expected calls is $O(1)$ per iteration.

\end{proof}

\subsection{Initial point reduction}\label{sec:initial-point-reduction}

In this section, we will show how to obtain an initial feasible point
for the algorithm.

\begin{defn}\label[defn]{def:ModConvProg}
Given a convex program $\min_{\ma\vx=\vb,\vx\in\mathcal{K}\subseteq \R^{d}}\vc^{\top}\vx$
and some $s > 0$, 
we define  $\vc_{1}=\vc,\vc_{2}=\vc_{3}=\frac{s \|\vc\|_2}{\sqrt{d}} \cdot \mathbf{1}$  and $\mathcal{P}=\{\vx^{(1)}\in\mathcal{K},(\vx^{(2)},\vx^{(3)})\in\R_{\geq0}^{2d}:\ma(\vx^{(1)}+\vx^{(2)}-\vx^{(3)})=\vb\}$. We then define the {\em modified convex program} by
\[
\min_{(\vx^{(1)},\vx^{(2)},\vx^{(3)})\in\mathcal{P}}\vc_{1}^{\top}\vx^{(1)}+\vc_{2}^{\top}\vx^{(2)}+\vc_{3}^{\top}\vx^{(3)}.
\]
 We denote $(\vc_{1},\vc_{2},\vc_{3})$ 
by $\overline{\vc}$.
\end{defn}

\begin{thm} 
\label{thm:initial-main}Given a convex program $\min_{\ma\vx=\vb,\vx\in\mathcal{K}\subseteq \R^d}\vc^{\top}\vx$
with outer radius $R$ and some convex set $\kin$ with $\kin\subseteq \kcal$ and inner radius $r$. For any modified convex program as in \cref{def:ModConvProg}  with $s\geq 48\nu t \sqrt{d}\cdot\frac{R}{r}\cdot\|c\|_{2}R$.
For an arbitrary $t\in \R_{\geq 0}$, we define the function \[f_{t}(\vx^{(1)},\vx^{(2)},\vx^{(3)})=t(\vc_{1}^{\top}\vx^{(1)}+\vc_{2}^{\top}\vx^{(2)}+\vc_{3}^{\top}\vx^{(3)})+\barr_{\mathcal{P}_{\text{in}}}(\vx^{(1)},\vx^{(2)},\vx^{(3)})\] 
where $\barr_{\mathcal{P}_{\text{in}}}$ is some $\nu$ self-concordant barrier for the set
\[\mathcal{P}_{\text{in}}=\{\vx^{(1)}\in\kin,(\vx^{(2)},\vx^{(3)})\in\R_{\geq0}^{2d}:\ma(\vx^{(1)}+\vx^{(2)}-\vx^{(3)})=\vb\}.\]
Given $\overline{\vx}_{t}\defeq(\vx_{t}^{(1)},\vx_{t}^{(2)},\vx_{t}^{(3)})=\arg\min_{(\vx^{(1)},\vx^{(2)},\vx^{(3)})\in\pin}f_{t}(\vx^{(1)},\vx^{(2)},\vx^{(3)})$, we denote $\vx_{\text{in}}=\vx_{t}^{(1)}+\vx_{t}^{(2)}-\vx_{t}^{(3)}$.
Suppose $\min_{\ovx \in\pin} \bar\vc^\top \ovx  \leq \min_{\ovx\in\mathcal{P}	}\bar\vc^\top \ovx + \gamma$, we have the following

\begin{enumerate}
\item $\ma\vx_{\text{in}}=\vb$,
\item $\vx_{\text{in}}\in\kin$,
\item $\vc^{\top}\vx_{\text{in}}\leq \min_{\vx\in\kcal,\ma\vx=\vb}\vc^{\top}\vx +\frac{\nu+1}{t}+\gamma$.
\end{enumerate}
\end{thm}

First, we show that $\vx_{t}^{(1)}$ is not too close to the boundary.
Before we proceed, we need the following lemmas.
\begin{lem}[Theorem 4.2.5~\cite{DBLP:books/sp/Nesterov04}]
\label[lem]{lem:barrier-gradient} Let $\barr$ be a $\nu$-self-concordant
barrier. Then, for any $\vx\in\textrm{dom}(\barr)$ and $\vy\in\textrm{dom}(\barr)$
such that 
\[
\langle\barr'(\vx),\vy-\vx\rangle\geq0,
\]
we have 
\[
\|\vy-\vx\|_{\vx}\leq\nu+2\sqrt{\nu}.
\]
\end{lem}

\begin{lem}[Theorem 2 of \cite{zong2022short}]
\label[lem]{lem:two-sided-ineq}Given a convex set \footnote{The original theorem is stated  only for polytopes, but their proof works for general convex sets.} $\Omega$ with a $\nu$-self-concordant
barrier $\barr_{\Omega}$ and inner radius $r$. Let $\vx_{t}=\arg\min_{\vx}t\cdot\vc^{\top}\vx+\barr_{\Omega}(\vx)$.
Then, for any $t>0$, 
\[
\min\left\{\frac{1}{2t},\frac{r \|c\|_{2}}{4\nu + 4 \sqrt{\nu}}\right\}\leq\vc^{\top}\vx_{t}-\vc^{\top}\vx_{\infty}\leq\frac{\nu}{t}.
\]
\end{lem}

Consider the optimization problem restricted in the subspace $\{(\vx^{(1)},\vx^{(2)},\vx^{(3)}):A(\vx^{(1)}+\vx^{(2)}-\vx^{(3)})=\vb\}$,
as a direct corollary of theorem above we have the following:
\begin{cor}
Let $\bar \vx_{t}$ be as the same as defined in \cref{thm:initial-main}.
For $t\geq\frac{4\nu}{r \|c\|_{2}}$, we have $\text{dist}(\vx_{t}^{(1)},\vx_{\infty}^{(1)})\geq\frac{1}{2t\|\vc\|_{2}}$.
\end{cor}

Now, we are ready to show $\text{dist}(\vx_{t}^{(1)},\partial\kin)$
is not too small.
\begin{thm}
Let $\bar\vx_{t}$ be the same as defined in \cref{thm:initial-main}.
For $t\geq\frac{4\nu}{r \|c\|_{2}}$, we have $\text{dist}(\vx_{t}^{(1)},\partial\kin)\geq\frac{r}{12 \nu t\|\vc\|_{2}R}$.
\end{thm}

\begin{proof}
We consider the domain restricted in the subspace $\{(\vx^{(1)},\vx^{(2)},\vx^{(3)}):\ma(\vx^{(1)}+\vx^{(2)}-\vx^{(3)})=\vb\}$.
By the optimality of $\bar\vx_{t}$ and \cref{lem:barrier-gradient}, we
have 
\[
\mathcal{K}_{\mathcal{H}}\subseteq\{\vx:\|\vx-\vx_{t}^{(1)}\|_{\vx_{t}^{(1)}}\leq\nu+2\sqrt{\nu}\},
\]
where $\h=\{\vx:\vc^{\top}(\vx_{t}^{(1)}-\vx)\geq0\}$ and $\kh\defeq\h\cap\kin.$ 

Recall that $\kin$ contains a ball of radius $r$, we denote it by
$B$. We note that $\conv(\vx_{\infty}^{(1)},B)$ is a union of a
ball and a convex cone $\mathcal{C}$ with diameter at most
$2R$. We observe that the set $\conv(\vx_{\infty}^{(1)},B)\cap\h$
contains a ball of radius at least $\frac{r}{4t\|\vc\|_{2}R}$ since
$\text{dist}(\vx_{\infty}^{(1)},\partial\h)\geq\frac{1}{2t\|\vc\|_{2}}$.

We note that 
\[
\conv(\vx_{\infty}^{(1)},B)\cap\h \subseteq \kin\subseteq \{\vx:\|\vx-\vx_{t}^{(1)}\|_{\vx_{t}^{(1)}}\leq\nu+2\sqrt{\nu}\},
\] this implies $\{\vx:\|\vx-\vx_{t}^{(1)}\|_{\vx_{t}^{(1)}}\leq\nu+2\sqrt{\nu}\}$ contains a ball of radius at least $\frac{r}{4t\|\vc\|_{2}R}$, and then by
\cref{thm:SCFactInit}, we have 
$B(\vx_{t}^{(1)},\frac{r}{4(\nu + 2\sqrt{\nu})t\|\vc\|_{2}R})\subseteq\kin.$
\end{proof}
\begin{lem}\label[lem]{lem:MaxDistBetweenCreatedPoints}
Let $(\vx_{t}^{(1)},\vx_{t}^{(2)},\vx_{t}^{(3)})\in \R^{3d}$ be the same as
defined in \cref{thm:initial-main}. If $t>\frac{\nu}{\|\vc\|_{2}R}$, then
we have $\|\vx_{t}^{(2)}-\vx_{t}^{(3)}\|_{2}\leq\frac{4\sqrt{d}}{s}R$.
\end{lem}

\begin{proof}
Let $\vx^{\star}_{\textrm{in}}=\arg\min_{\vx\in\kin,\ma\vx=\vb}\vc^{\top}\vx$ and
$\ovx^{\star}_{\textrm{in}}=\arg\min_{\vx\in\pin}\overline{\vc}^{\top}\ovx$. Since $\vx^{\star}\in \ball(0,R)$, we have
\[
\vc^{\top}\vx^{\star}_{\textrm{in}}\leq\|\vc\|_{2}R.
\]
Note that $(\vx^{\star}_{\textrm{in}},\mathbf{0},\mathbf{0})\in\pin$, this
means we have 
\[
\overline{\vc}^{\top}\ovx^{\star}_{\textrm{in}}\leq\vc^{\top}\vx^{\star}_{\textrm{in}}\leq\|\vc\|_{2}R.
\]
Combining this with the second inequality in \cref{lem:two-sided-ineq}, we get 

\[
\overline{\vc}^{\top}\overline{\vx}_{t}\leq\overline{\vc}^{\top}\ovx^{\star}_{\textrm{in}}+\frac{\nu}{t}\leq\|\vc\|_{2}R+\frac{\nu}{t}\leq2\|\vc\|_{2}R.
\]
We further note that
\[
\vc_2^\top \vx_t^{(2)}\leq \overline{\vc}^{\top}\overline{\vx}_{t}  \leq  2\|\vc\| R.
\]
This shows
\[
\max\{\|\vx_{t}^{(2)}\|_{2},\|\vx_{t}^{(3)}\|_{2}\}\leq\frac{2\sqrt{d}\|\vc\|_{2}R}{\|\vc\|_2s}\leq\frac{2 \sqrt{d}R}{s}.
\]

Hence, we have 
\[
\|\vx_{t}^{(2)}-\vx_{t}^{(3)}\|_{2}\leq\frac{4 \sqrt{d}}{s}R.
\]
\end{proof}
Now, we are ready to prove \cref{thm:initial-main}.
\begin{proof}[Proof of \cref{thm:initial-main}]
We note that $\vx_{\text{in}}$ satisfies (1), directly follows by
definition of $\mathcal{P}$. By assumption, we have $s\geq 48\nu t \sqrt{d}\cdot\frac{R}{r}\cdot\|c\|_{2}R$; using this in \cref{lem:MaxDistBetweenCreatedPoints}, we have 
\[
\|\vx_{t}^{(2)}-\vx_{t}^{(3)}\|_{2}\leq\frac{r}{12 \nu t\|\vc\|_{2}R}.
\]
This means $\vx_{\text{in}}=\vx_{t}^{(1)}+\vx_{t}^{(2)}-\vx_{t}^{(3)}\in\kin$
since $\text{dist}(\vx_{t}^{(1)},\partial\kin)\geq\frac{r}{12 \nu t\|\vc\|_{2}R}$.

Now, we show $\vc^{\top}\vx_{\text{in}}$ is close to $\vc^{\top}\vx^{\star}$.

Let $\vx^{\star}=\arg\min_{\vx\in\kcal,\ma\vx=\vb}\vc^{\top}\vx$ and $\ovx^{\star}=\arg\min_{\vx\in\mathcal{P}}\overline{\vc}^{\top}\ovx$.
By \cref{lem:two-sided-ineq}, we have
\[
\overline{\vc}^{\top}\overline{\vx}_{t}-\frac{\nu}{t}\leq\overline{\vc}^{\top}\overline{\vx}^{\star}_{\textrm{in}}\leq \overline{\vc}^\top \ovx^\star +\gamma\leq\vc^{\top}\vx^{\star}+\gamma.
\]
This implies 
\[
\vc^{\top}\vx_{t}^{(1)}\leq\overline{\vc}^{\top}\overline{\vx}_{t}\leq\vc^{\top}\vx^{\star}+\frac{\nu}{t}+\gamma.
\]
We have
\[
\vc^{\top}\vx_{\text{in}}=\vc^{\top}(\vx_{t}^{(1)}+\vx_{t}^{(2)}-\vx_{t}^{(3)})\leq\vc^{\top}\vx^{\star}+\frac{\nu}{t}+\frac{4}{s}\|\vc\|_{2}R\leq\vc^{\top}\vx^{\star}+\frac{\nu + 1}{t}+\gamma.
\]

\end{proof}

\section{Acknowledgements}
We thank Ian Covert for helpful discussions about the problem applications and Ewin Tang for helpful feedback on the paper. 

\newpage

\printbibliography

\newpage
\begin{appendix}
\section{Decomposable submodular function minimization}\label{sec:decompSFM} 
\subsection{Preliminaries}

Throughout, $V$ denotes the ground set of elements. 
A set function $f: 2^V \rightarrow \mathbb{R}$ is {\em submodular} if it satisfies the following {\em diminishing marginal differences} property: 

\begin{defn}[Submodularity] A function $f: 2^{V} \rightarrow \mathbb{R}$ is submodular if $f(T \cup \{i\}) - f(T) \leq f(S \cup \{i\} ) - f(S)$, for any subsets $S \subseteq T \subseteq V$ and $i \in V \setminus T$. 
\end{defn}

We may assume without loss of generality that $f(\emptyset) = 0$ by replacing $f(S)$ by $f(S) - f(\emptyset)$. 
We assume that $f$ is accessed by an {\em evaluation oracle} and use $\EO$ to denote the time to compute $f(S)$ for a subset $S$. 
Our algorithm for decomposable SFM is based on the Lov\'asz extension \cite{gls88}, a standard convex extension of a submodular function.  

\begin{defn}[Lov\'asz extension \cite{gls88}]
The Lov\'asz extension $\hat{f}:[0,1]^V \rightarrow \mathbb{R}$ of a submodular function $f$ is defined as 
\begin{align*}
\hat{f}(x) = \E_{t \sim [0,1]} [f(\{i \in V: x_i \geq t\})],
\end{align*}
where $t \sim [0,1]$ is drawn uniformly at random from $[0,1]$. 
\end{defn}

The Lov\'asz extension $\hat{f}$ of a submodular function $f$ has many desirable properties. 
In particular, $\hat{f}$ is a convex relaxation of $f$ and it can be evaluated efficiently. 

\begin{thm}[Properties of Lov\'asz extension \cite{gls88}]
\label{thm:lovasz_extension_properties}
Let $f: 2^{V} \rightarrow \mathbb{R}$ be a submodular function and $\hat{f}$ be its Lov\'asz extension. Then, 
\begin{itemize}
	\vspace{-0.2cm}
	\item [(a)] $\hat{f}$ is convex and $\min_{x \in [0,1]^V} \hat{f}(x) = \min_{S \subseteq V} f(S)$;
	\vspace{-0.12cm}
	\item [(b)] $f(S) = \hat{f}(I_S)$ for any subset $S \subseteq V$, where $I_S$ is the indicator vector for $S$; 
	\vspace{-0.12cm}
	\item [(c)] Suppose $x \in [0,1]^V$ satisfies $x_1 \geq \cdots \geq x_{|V|}$, then $\hat{f}(x) = \sum_{i=1}^{|V|} (f([i]) - f([i-1])) x_i$.
\end{itemize}
\vspace{-0.2cm}
\end{thm}

Property (c) in \Cref{thm:lovasz_extension_properties} allows us to implement a sub-gradient oracle for $\hat{f}$ by evaluating $f$.

\begin{thm}[Sub-gradient oracle implementation for Lov\'asz extension, Theorem 61 of~\cite{lsw15}] \label{thm:SO_from_EO}
Let $f: 2^{V} \rightarrow \mathbb{R}$ be a submodular function and $\hat{f}$ be its Lov\'asz extension. 
Then a sub-gradient for $\hat{f}$ can be implemented in time $O(|V| \cdot \EO + |V|^2)$. 
\end{thm}

\subsection{Decomposable submodular function minimization proofs}

In this subsection, we prove the following more general version of \Cref{thm:decompSFM_intro}.

\begin{thm}[Decomposable SFM] \label{thm:decompSFM_body}
Let $F: V \rightarrow [-1,1]$ be given by $F(S) = \sum_{i = 1}^n F_i(S \cap V_i)$, where each $F_i: 2^{V_i} \rightarrow \mathbb{R}$ is a submodular function on $V_i \subseteq V$ with $|V_i| = d_i$. Let $m = \sum_{i=1}^n d_i$ and $d_{\max} := \max_{i \in [n]} d_i$. Then we can find an $\epsilon$-approximate minimizer of $f$ using at most $O(d_{\max} m \log(m/\epsilon))$ evaluation oracle calls. 
\end{thm} 

\begin{proof}
Let $\hat{f_i}$ be the Lov\'asz extension of each $f_i$, then $\hat{f} = \sum_{i=1}^n \hat{f}_i$ is the Lov\'asz extension of $f$.
Note that $\hat{f}$ is $2$-Lipschitz since the range of $f$ is $[-1,1]$. Also, the diameter of the range $[0,1]^{V_i}$ for each Lov\'asz extension $\hat{f}_i$ is at most $\sqrt{|V_i|} \leq \sqrt{d_{\max}}$. Thus using \Cref{thm:mainResult}, we can find a vector $x \in [0,1]^V$ such that $\hat{f}(x) \leq \min_{x^* \in [0,1]^V} \hat{f}(x^*) + \epsilon$ in $\poly(m \log(1/\epsilon))$ time and $O(m \log(m \sqrt{d_{\max}} /\epsilon)) = O(m \log(m /\epsilon))$ subgradients of the $\hat{f}_i$'s. By \Cref{thm:SO_from_EO}, each sub-gradient of $\hat{f}_i$ can be computed by making at most $d_i \leq d_{\max}$ queries to the evaluation oracle for $f_i$. Thus the total number of evaluation oracle calls we make in finding an $\epsilon$-additive approximate minimizer $x \in [0,1]^V$ of $\hat{f}$ is at most $O(d_{\max} m \log(m /\epsilon))$.

Next we turn the $\epsilon$-additive approximate minimizer $x$ of $\hat{f}$ into an $\epsilon$-additive approximate minimizer $S \subseteq V$ for $f$. 
Without loss of generality, assume that $x_1 \geq \cdots \geq x_{|V|}$. Then by property (c) in \Cref{thm:lovasz_extension_properties}, we have 
\[
\hat{f}(x) = \sum_{i=1}^{|V|} (f([i]) - f([i-1])) x_i = f(V) \cdot x_{|V|} +  \sum_{i=1}^{|V|-1} f([i]) \cdot (x_i - x_{i+1}) .
\]
Since $x_ i -x_{i+1} \geq 0$, the above implies that $\min_{i \in \{1,\dots,|V|\}} f([i]) \leq \hat{f}(x)$. Thus we can find a subset $S \subseteq V$ among $f([i])$ for all $i\in \{1, \cdots, |V|\}$ such that $f(S) \leq \hat{f}(x)$. Then by property (a) in \Cref{thm:lovasz_extension_properties}, the set $S$ is an $\epsilon$-additive approximate minimizer of $f$. This proves the theorem. 
\end{proof}

\end{appendix}

\end{document}